\documentclass[reqno,12pt]{amsart}
\usepackage{amscd,amsfonts,amssymb}
\numberwithin{equation}{section}
\newtheorem{Theorem}{Theorem}[section]
\newtheorem{Lemma}{Lemma}[section]
\newtheorem{Corollary}{Corollary}[section]
\newtheorem{Proposition}{Proposition}[section]
\theoremstyle{definition}

\theoremstyle{remark}

\newtheorem{Example}{Example}[section]

\newcommand{\esssup}{\mathop{\rm ess \, sup}\limits}
\newcommand{\mes}{\mathop{\rm mes}\nolimits}
\newcommand{\diver}{\mathop{\rm div}\nolimits}
\author{Andrej A. Kon'kov}
\address{Department of Differential Equations,
Faculty of Mechanics and Mathematics,
Mo\-s\-cow Lo\-mo\-no\-sov State University,
Vorobyovy Gory,
119992 Moscow, Russia}
\email{konkov@mech.math.msu.su}
\title[]{On comparison theorems for elliptic inequalities}
\thanks{The research was supported by RFBR, grant 09-01-12157.}
\date{}
\begin{document}

\maketitle

\section{Introduction}

Suppose that $\Omega$ is a non-empty open subset of ${\mathbb R}^n$, $n \ge 2$.
Let us denote:
$
	\Omega_{R_0, R_1}
	=
	\{
		x \in \Omega : R_0 < |x| < R_1
	\}
$
and
$
	\Gamma_{R_0, R_1}
	=
	\{
		x \in \partial \Omega : R_0 < |x| < R_1
	\}.
$
By $B_r^x$ we mean the open ball in ${\mathbb R}^n$ 
of radius $r > 0$ and center at a point $x$.
Also put 
$S_r^x = \partial B_r^x$.
In the case of $x = 0$, we write $B_r$ and $S_r$
instead of $B_r^0$ and $S_r^0$, respectively.

Consider the inequality
\begin{equation}
	\diver A (x, D u)
	\ge
	F (x, u, D u)
	\quad
	\mbox{in }
	\Omega_{R_0, R_1},
	\;
	0 \le R_0 < R_1 \le \infty,
	\label{1.1}
\end{equation}
where
$D = (\partial / \partial x_1,\ldots,\partial / \partial x_n)$
is the gradient operator and
$A : \Omega_{R_0, R_1} \times {\mathbb R}^n \to {\mathbb R}^n$
is s measurable function such that
$$
	C_1
	|\xi|^p
	\le
	\xi
	A (x, \xi),
	\quad
	|A (x, \xi)|
	\le
	C_2
	|\xi|^{p-1}
$$
with some constants
$C_1 > 0$,
$C_2 > 0$,
and
$p > 1$
for almost all
$x \in \Omega_{R_0, R_1}$
and for all
$\xi \in {\mathbb R}^n$.
We say that $u$ is a solution of~\eqref{1.1} if
$
	u
	\in
	W_p^1 (
		\Omega_{R_0, r}
	)
	\cap
	L_\infty (
		\Omega_{R_0, r}
	),
$
$
	A (x, D u)
	\in
	L_{p/(p-1)} (
		\Omega_{R_0, r}
	),
$
and
$
	F (x, u, D u)
	\in
	L_{p/(p-1)} (
		\Omega_{R_0, r}
	)
$
for any real number $r \in (R_0, R_1)$ and, moreover,
$$
	- \int_{
		\Omega_{R_0, R_1}
	}
	A (x, D u)
	D \varphi
	\, dx
	\ge
	\int_{
		\Omega_{R_0, R_1}
	}
	F (x, u, D u)
	\varphi
	\, dx
$$
for any non-negative function
$
	\varphi 
	\in 
	C_0^\infty (
		\Omega_{R_0, R_1}
	)
$~\cite{LU}.
In so doing, the condition
\begin{equation}
	\left.
		u
	\right|_{
		\Gamma_{R_0, R_1}
	}
	=
	0
	\label{1.3}
\end{equation}
means that
$
	\varphi u
	\in
	{
		\stackrel{\rm \scriptscriptstyle o}{W}\!\!{}_p^1
		(
			\Omega_{R_0, R_1}
		)
	}
$
for any
$
	\varphi
	\in
	C_0^\infty 
	(
		B_{R_0, R_1}
	),
$
where
$
	B_{R_0, R_1}
	=
	\{
		x \in {\mathbb R}^n : R_0 < |x| < R_1
	\}.
$
In particular, if $\Omega = {\mathbb R}^n$,
then~\eqref{1.3} is fulfilled for all
$
	u
	\in
	W_{p, loc}^1
	(
		B_{R_0, R_1}
	).
$

Throughout this paper, we assume that
$S_r \cap \Omega \ne \emptyset$
for any $r \in (R_0, R_1)$.
Let $u$ be a solution of~\eqref{1.1}, \eqref{1.3}. 
Put
\begin{equation}
	M (r; u)
	=
	\esssup_{
		S_r
		\cap
		\Omega
	}
	u,
	\quad
	r \in (R_0, R_1),
	\label{1.4}
\end{equation}
where the restriction of $u$ to
$
	S_r
	\cap
	\Omega
$
is understood in the sense of the trace and
the $\esssup$ in the right-hand side of~\eqref{1.4} is with respect to
$(n-1)$-dimensional Lebesgue measure on $S_r$.
We also assume that the right-hand side of inequality~\eqref{1.1} 
satisfies the following condition:
there exist a real number $\sigma > 1$ and locally bounded measurable functions 
$f : [R_0, R_1) \times (0, \infty) \to [0,\infty)$
and
$b : [R_0, R_1) \to [0, \infty)$
such that
$$
	f (r, t-0) = f (r, t)
	\quad
	\mbox{for all } R_0 < r < R_1, \:  t > 0,
$$
$$
	f (r, t_1) \ge f (r, t_2)
	\quad
	\mbox{for all } R_0 < r < R_1, \: t_1 \ge t_2 > 0
$$
and, moreover,
\begin{equation}
	F (x, t, \xi)
	\ge
	\sup_{
		r
		\in
		(|x| / \sigma, \sigma |x|)
		\cap
		(R_0, R_1)
	}
	f (r, t)
	-
	|\xi|^{p - 1}
	\inf_{
		r
		\in
		(|x| / \sigma, \sigma |x|)
		\cap
		(R_0, R_1)
	}
	b (r)
	\label{1.5}
\end{equation}
for almost all
$x \in \Omega_{R_0, R_1}$
and for all
$t \in (0, \infty)$
and
$\xi \in {\mathbb R}^n$.

The questions studied in this article were earlier investigated by a number
of authors~\cite{Keller}--\cite{Me}, \cite{MP}--\cite{VeronBook}.
Our aim is to estimate the function $M (\cdot; u)$
by a solution of an ordinary differential equation,
which contains the radial $p$-Laplace operator with the lowest terms.

\section{Main results}\label{S2}

\begin{Theorem}\label{T2.1}
Let $u$ be a non-negative solution of problem~\eqref{1.1}, \eqref{1.3}
such that ${M (\cdot; u)}$ is a non-decreasing function on the
interval $(R_0, R_1)$ with
\begin{equation}
	M (R_0 + 0; u) > 0.
	\label{T2.1.1}
\end{equation}
Then for all real numbers $a > p - 2$ and $k > 0$ 
there exist constants $\alpha > 0$ and $\beta > 0$
depending only on
$n$, $p$, $a$, $k$, $\sigma$, $C_1$, and $C_2$
such that the Cauchy problem
\begin{equation}
	\frac{1}{r^{1+a}}
	\frac{d}{dr}
	\left(
		r^{1+a}
		\left|
			\frac{dm}{dr}
		\right|^{p-2}
		\frac{dm}{dr}
	\right)
	+
	k
	b (r)
	\left|
		\frac{dm}{dr}
	\right|^{p-2}
	\frac{dm}{dr}
	=
	\alpha
	f (r, \beta m),
	\label{T2.1.2}
\end{equation}
\begin{equation}
	m (R_0) = M (R_0+0; u),
	\quad
	m' (R_0) = 0,
	\label{T2.1.3}
\end{equation}
has a solution on $[R_0, R_1)$ satisfying the estimate
$$
        M (r; u)
        \ge
        m (r)
        >
        0
$$
for any $r \in (R_0, R_1)$.
\end{Theorem}

\begin{Theorem}\label{T2.2}
Under the hypotheses of Theorem~$\ref{T2.1}$, for all real numbers 
$a > p -2$ and $k > 0$ there exist constants $\alpha > 0$ and $\beta > 0$
depending only on $n$, $p$, $a$, $k$, $\sigma$, $C_1$, and $C_2$
such that
\begin{align}
	&
	M (r; u) - M (R_0 + 0; u)
	\nonumber
	\\
	&
	\quad
	\ge
	\int_{R_0}^r
	dt
	\,
	\left(
		\frac{
			\alpha
		}{
			t^{1+a}
		}
		\int_{R_0}^t
		\xi^{1+a}
		e^{
			- k
			\int_\xi^t
			b (\zeta)
			\,
			d\zeta
		}
		f (\xi, \beta M (\xi; u))
		\,
		d\xi
	\right)^{1 / (p - 1)}
	\label{T2.2.1}
\end{align}
for any $r \in (R_0, R_1)$.
\end{Theorem}

\begin{Example}\label{E2.1}
Consider the inequality
$$
	\sum_{i,j=1}^n 
	\frac{\partial}{\partial x_i}
	\left(
		a_{ij} (x)
		\frac{\partial u}{\partial x_j}
	\right)
	+
	\sum_{i=1}^n
	b_i (x)
	\frac{\partial u}{\partial x_i}
	\ge
	c (x, u)
$$
for a linear uniformly elliptic operator with locally bounded measurable coefficients.
Setting $p = 2$ and
$$
	F (x, t, \xi) 
	= 
	c (x, t) 
	- 
	\sum_{i = 1}^n
	b_i (x)
	\xi_i,
$$
one can show that relation~\eqref{1.5} is filfilled 
if $b$ and $f$ are non-negative functions
such that
$$
	b(r)
	\ge
	\sup_{
		x
		\in
		\Omega_{r / \sigma, r \sigma}
		\cap
		\Omega_{R_0, R_1}
	}
	\sum_{i=1}^n
	|b_i (x)|
	\quad
	\mbox{for all } r \in (R_0, R_1)
$$
and
$$
	f(r,t)
	\le
	\inf_{
		x 
		\in
		\Omega_{r / \sigma, r \sigma}
		\cap
		\Omega_{R_0, R_1}
	}
	c(x,t)
	\quad
	\mbox{for all } r \in (R_0, R_1), \: t \in (0,\infty).
$$
In this case, equation~\eqref{T2.1.2} takes the form
\begin{equation}
	\frac{
		d^2 m
	}{
		d r^2
	}
	+ 
	\left(
		\frac{
			1 + a
		}{
			r
		}
		+
		k
		b(r)
	\right)
	\frac{
		d m
	}{
		d r
	}
	=
	\alpha
	f(r, \beta m).
	\label{E2.1.1}
\end{equation}

Putting $a = n - 2$ and $k = 1$, we obviously obtain the radial part of the operator
$
 \triangle + b(|x|) D |x| D
$
in the left-hand side of~\eqref{E2.1.1}.
\end{Example}

\begin{proof}[Proof of Theorem~$\ref{T2.1}$]
Assume that Theorem~\ref{T2.2} is already proved.
Let us construct a sequence of maps
$m_i  : [R_0, R_1) \to (0, \infty)$
by setting
$m_0 (r) = M (R_0 + 0; u)$
and
$$
	m_i (r)
	=
	M (R_0+0; u)
	+
	\int_{R_0}^r
	dt
	\,
	\left(
		\frac{
			\alpha
		}{
			t^{1+a}
		}
		\int_{R_0}^t
		\xi^{1+a}
		e^{
			- k
			\int_\xi^t
			b (\zeta)
			\,
			d\zeta
		}
		f (\xi, \beta m_{i - 1} (\xi))
		\,
		d\xi
	\right)^{1 / (p - 1)}
$$
$i = 1,2,\ldots$.
We have
$M (r; u) \ge m_i (r) \ge m_{i-1} (r)$
for all $r \in (R_0, R_1)$,
$i = 1, 2, \ldots$.
Therefore, there exists a map
$m : [R_0, R_1) \to (0, \infty)$
such that $m_i$ tends to  $m$ everywhere on the interval $[R_0, R_1)$
as $i \to \infty$.

It is obvious that
${M (r; u)} \ge m (r)$
for all $r \in (R_0, R_1)$.
In addition, the following integral equation is valid:
$$
	m (r)
	=
	M (R_0+0; u)
	+
	\int_{R_0}^r
	dt
	\,
	\left(
		\frac{
			\alpha
		}{
			t^{1+a}
		}
		\int_{R_0}^t
		\xi^{1+a}
		e^{
			- k
			\int_\xi^t
			b (\zeta)
			\,
			d\zeta
		}
		f (\xi, \beta m (\xi))
		\,
		d\xi
	\right)^{1 / (p - 1)}
$$

Thus, to complete the proof it remains to verify by direct differentiation that
$m$ is a solution of problem~\eqref{T2.1.2}, \eqref{T2.1.3}.
\end{proof}

\section{Proof of Theorem~\ref{T2.2}}

From now on we assume that $a > p - 2$ and $k > 0$ are some fixed real numbers 
and $u \ge 0$ is a solution of problem~\eqref{1.1}, \eqref{1.3} such that 
${M (\cdot; u)}$ is a non-decreasing function on the interval $(R_0, R_1)$
satisfying condition~\eqref{T2.1.1}.
Without loss of generality it can also be assumed that 
$$
	\inf_{
		(R_0, R_1)
	}
	b 
	> 
	0;
$$
otherwise we prove~\eqref{T2.2.1} with $b$ replaced by 
$b + \delta$, where $\delta$ is a positive real number, and let $\delta$
tend to zero afterwards.

From the maximum principle, it follows that
\begin{equation}
	M (r - 0; u) = M (r; u),
	\quad
	r \in (R_0, R_1),
	\label{3.1}
\end{equation}
(see Corollary~\ref{C4.1}, Section~\ref{S4}).

\begin{Lemma}\label{L3.1}
Let $0 < \beta < 1$, $R_0 < r_0 < r_1 < R_1$, and $\sigma^2 r_0 \ge r_1$.
If $\beta^{1/2} M (r_1; u) \le M (r_0; u)$, then
$$
	M (r_1; u) - M (r_0; u)
	\ge
	\gamma_1
	\min
	\left\{
		(r_1 - r_0)^{p / (p-1)},
		\frac{
			r_1 - r_0
		}{
			\lambda^{1 / (p - 1)}
		}
	\right\}
	f^{1 / (p-1)} (s, \beta M (r_1; u))
$$
for all $s \in [r_1 / \sigma, \sigma r_0] \cap (R_0, R_1)$,
where 
$$
	\lambda
	=
	\inf_{
		[r_1 / \sigma, \sigma r_0]
		\cap
		(R_0, R_1)
	}
	b
$$
and the constant $\gamma_1 > 0$ depends only on
$n$, $p$, $C_1$, $C_2$, and $\beta$.
\end{Lemma}

The proof of Lemma~\ref{L3.1} is given in Section~\ref{S4}.

\begin{Corollary}\label{C3.1}
Suppose that
$0 < \beta < 1$,
$R_0 < r_0 < r_1 < R_1$,
$\sigma r_0 \ge r_1$
and, moreover,
$\beta^{1/2} M (r_1; u) \le  M (r_0; u)$.
Then
\begin{equation}
	M (r_1; u) - M (r_0; u)
	\ge
	\gamma_2
	(r_1 - r_0)
	\left(
		\int_{\rho_0}^{\rho_1}
		e^{
			- k
			\int_\xi^{\rho_1}
			b (\zeta)
			\,
			d\zeta
		}
		f (\xi, \beta M (\xi; u))
		\,
		d\xi
	\right)^{1 / (p - 1)}
	\label{C3.1.1}
\end{equation}
for all real numbers
$R_0 < \rho_0 < \rho_1 < R_1$
satisfying the inequalities
$r_1 / \sigma \le \rho_0$,
$\rho_1 \le r_1$,
and
$\rho_1 - \rho_0 \le r_1 - r_0$,
where the constant $\gamma_2 > 0$ depends only on
$n$, $p$, $k$, $C_1$, $C_2$, and $\beta$.
\end{Corollary}

\begin{proof}
We have
\begin{align*}
	\int_{
		\rho_0
	}^{
		\rho_1
	}
	e^{
		- k
		\int_\xi^{
			\rho_1
		}
		b (\zeta)
		\,
		d\zeta
	}
	f (\xi, \beta M (\xi; u))
	\,
	d\xi
	&
	\le
	f (\xi_*, \beta M (\xi_*; u))
	\int_{
		\rho_0
	}^{
		\rho_1
	}
	e^{
		- k
		\int_\xi^{
			\rho_1
		}
		b (\zeta)
		\,
		d\zeta
	}
	\,
	d\xi
	\\
	&
	\le
	f (\xi_*, \beta M (\xi_*; u))
	\int_{
		\rho_0
	}^{
		\rho_1
	}
	e^{
		- k
		\lambda
		(\rho_1 - \xi)
	}
	\,
	d\xi
\end{align*}
for some $\xi_* \in (\rho_0, \rho_1)$, where
$$
	\lambda
	=
	\inf_{
		(\rho_0, \rho_1)
	}
	b.
$$
Since
$$
	\int_{
		\rho_0
	}^{
		\rho_1
	}
	e^{
		- k
		\lambda
		(\rho_1 - \xi)
	}
	\,
	d\xi
	=
	\frac{
		1 
		- 
		e^{
			- k
			\lambda
			(\rho_1 - \rho_0)
		}
	}{
		k
		\lambda
	}
	\le
	\min
	\left\{
		\rho_1 - \rho_0,
		\frac{
			1
		}{
			k
			\lambda
		}
	\right\},
$$
this implies the estimate
\begin{align*}
	&
	(r_1 - r_0)
	\left(
		\int_{\rho_0}^{\rho_1}
		e^{
			- k
			\int_\xi^{\rho_1}
			b (\zeta)
			\,
			d\zeta
		}
		f (\xi, \beta M (\xi; u))
		\,
		d\xi
	\right)^{1 / (p - 1)}
	\\
	&
	\quad
	\le
	\min
	\left\{
		(r_1 - r_0)^{p / (p - 1)},
		\frac{
		r_1 - r_0
		}{
			(k \lambda)^{1 / (p - 1)}
		}
	\right\}
	f^{1 / (p - 1)} (\xi_*, \beta M (\xi_*; u)),
\end{align*}
whence in accordance with Lemma~\ref{C3.1} we obtain~\eqref{C3.1.1}.
\end{proof}

\begin{Corollary}\label{C3.2}
Let the conditions of Corollary~$\ref{C3.1}$ be fulfilled, then
\begin{align*}
	&
	M (r_1; u) - M (r_0; u)
	\\
	&
	\quad
	\ge
	\gamma_3
	(r_1 - r_0)
	\left(
		\frac{
			r_0 - \rho_1
		}{
			\rho_1 - \rho_0
		}
		\int_{\rho_0}^{\rho_1}
		e^{
			- k
			\int_\xi^{r_0}
			b (\zeta)
			\,
			d\zeta
		}
		f (\xi, \beta M (\xi; u))
		\,
		d\xi
	\right)^{1 / (p - 1)}
\end{align*}
for all real numbers
$R_0 < \rho_0 < \rho_1 < R_1$
satisfying the inequalities
$r_1 / \sigma \le \rho_0$,
$\rho_1 < r_0$,
and
$r_0 - \rho_1 \le r_1 - r_0$,
where the constant $\gamma_3 > 0$ depends only on
$n$, $p$, $k$, $C_1$, $C_2$, and $\beta$.
\end{Corollary}

\begin{proof}
There exists $\xi_* \in (\rho_0, \rho_1)$ such that
\begin{align}
	\int_{
		\rho_0
	}^{
		\rho_1
	}
	e^{
		- k
		\int_\xi^{
			r_0
		}
		b (\zeta)
		\,
		d\zeta
	}
	f (\xi, \beta M (\xi; u))
	\,
	d\xi
	&
	\le
	f (\xi_*, \beta M (\xi_*; u))
	\int_{
		\rho_0
	}^{
		\rho_1
	}
	e^{
		- k
		\int_\xi^{
			r_0
		}
		b (\zeta)
		\,
		d\zeta
	}
	\,
	d\xi
	\nonumber
	\\
	&
	\le
	f (\xi_*, \beta M (\xi_*; u))
	\int_{
		\rho_0
	}^{
		\rho_1
	}
	e^{
		- k
		\lambda
		(r_0 - \xi)
	}
	\,
	d\xi,
	\label{PC3.2.1}
\end{align}
where
$$
	\lambda
	=
	\inf_{
		(\rho_0, r_0)
	}
	b.
$$
We have
$$
	\int_{
		\rho_0
	}^{
		\rho_1
	}
	e^{
		- k
		\lambda
		(r_0 - \xi)
	}
	\,
	d\xi
	=
	\frac{
		e^{
			- k
			\lambda
			(r_0 - \rho_1)
		}
		-
		e^{
			- k
			\lambda
			(r_0 - \rho_0)
		}
	}{
		k \lambda
	}
	\le
	(\rho_1 - \rho_0)
	e^{
		- k
		\lambda
		(r_0 - \rho_1)
	}.
$$
In addition,
\begin{align*}
	e^{
		- k
		\lambda
		(r_0 - \rho_1)
	}
	&
	\le
	\frac{
		1
		-
		e^{
			- k
			\lambda
			(r_0 - \rho_1)
		}
	}{
		k
		\lambda
		(r_0 - \rho_1)
	}
	\\
	&
	\le
	\frac{1}{r_0 - \rho_1}
	\min
	\left\{
		r_0 - \rho_1,
		\frac{
			1
		}{
			k
			\lambda
		}
	\right\}
	\\
	&
	\le
	\frac{1}{r_0 - \rho_1}
	\min
	\left\{
		r_1 - r_0,
		\frac{
			1
		}{
			k
			\lambda
		}
	\right\}.
\end{align*}
Hence, we obtain
$$
	\int_{
		\rho_0
	}^{
		\rho_1
	}
	e^{
		- k
		\lambda
		(r_0 - \xi)
	}
	\,
	d\xi
	\le
	\frac{\rho_1 - \rho_0}{r_0 - \rho_1}
	\min
	\left\{
		r_1 - r_0,
		\frac{
			1
		}{
			k
			\lambda
		}
	\right\}.
$$
The last formula and~\eqref{PC3.2.1} imply the estimate
\begin{align*}
	&
	(r_1 - r_0)
	\left(
		\frac{
			r_0 - \rho_1
		}{
			\rho_1 - \rho_0
		}
		\int_{\rho_0}^{\rho_1}
		e^{
			- k
			\int_\xi^{r_0}
			b (\zeta)
			\,
			d\zeta
		}
		f (\xi, \beta M (\xi; u))
		\,
		d\xi
	\right)^{1 / (p - 1)}
	\\
	&
	\quad
	\le
	\min
	\left\{
		(r_1 - r_0)^{p / (p - 1)},
		\frac{
		r_1 - r_0
		}{
			(k \lambda)^{1 / (p - 1)}
		}
	\right\}
	f^{1 / (p - 1)} (\xi_*, \beta M (\xi_*; u)).
\end{align*}
Thus, to complete the proof it remains to use Lemma~\ref{L3.1}.
\end{proof}

\begin{Lemma}\label{L3.2}
Suppose that
$0 < \beta < 1$,
$R_0 < r_0 < r_1 < R_1$,
and
$\beta^{1/2} M (r_1; u) \le  M (r_0; u)$,
then
\begin{align}
	&
	M (r_1; u) - M (r_0; u)
	\nonumber
	\\
	&
	\quad
	\ge
	\gamma_4
	\int_{r_0}^{r_1}
	dt
	\,
	\left(
		\frac{
			1
		}{
			t^{1+a}
		}
		\int_{r_0}^t
		\xi^{1+a}
		e^{
			- k
			\int_\xi^t
			b (\zeta)
			\,
			d\zeta
		}
		f (\xi, \beta M (\xi; u))
		\,
		d\xi
	\right)^{1 / (p - 1)},
	\label{L3.2.1}
\end{align}
where the constant $\gamma_4 > 0$ depends only on
$n$, $p$, $a$, $k$, $\sigma$, $C_1$, $C_2$, and $\beta$.
\end{Lemma}

\begin{proof}
In the case of $\sigma r_0 \ge r_1$, taking $\xi_* \in (r_0, r_1)$ 
such that
$$
	\esssup_{
		\xi \in (r_0, r_1)
	}
	f (\xi, \beta M (\xi; u))
	\le
	2
	f (\xi_*, \beta M (\xi_*; u)),
$$
we have
\begin{align*}
	&
	\int_{r_0}^{r_1}
	dt
	\,
	\left(
		\frac{
			1
		}{
			t^{1+a}
		}
		\int_{r_0}^t
		\xi^{1+a}
		e^{
			- k
			\int_\xi^t
			b (\zeta)
			\,
			d\zeta
		}
		f (\xi, \beta M (\xi; u))
		\,
		d\xi
	\right)^{1 / (p - 1)}
	\\
	&
	\quad
	\le
	2^{1 / (p - 1)}
	f^{1 / (p-1)} (\xi_*, \beta M (\xi_*; u))
	\int_{r_0}^{r_1}
	dt
	\,
	\left(
		\frac{
			1
		}{
			t^{1+a}
		}
		\int_{r_0}^t
		\xi^{1+a}
		e^{
			- k
			\int_\xi^t
			b (\zeta)
			\,
			d\zeta
		}
		\,
		d\xi
	\right)^{1 / (p - 1)}
	\\
	&
	\quad
	\le
	2^{1 / (p - 1)}
	f^{1 / (p-1)} (\xi_*, \beta M (\xi_*; u))
	\int_{r_0}^{r_1}
	dt
	\,
	\left(
		\int_{r_0}^t
		e^{
			- k
			\lambda
			(t - \xi)
		}
		\,
		d\xi
	\right)^{1 / (p - 1)},
\end{align*}
where
$$
	\lambda
	=
	\inf_{
		(r_0, r_1)
	}
	b.
$$
It presents no special problems to verify that
\begin{align*}
	\int_{r_0}^{r_1}
	dt
	\,
	\left(
		\int_{r_0}^t
		e^{
			- k
			\lambda
			(t - \xi)
		}
		\,
		d\xi
	\right)^{1 / (p - 1)}
	&
	=
	\int_{r_0}^{r_1}
	dt
	\,
	\left(
		\frac{
			1
			-
			e^{
				- k
				\lambda
				(t - r_0)
			}
		}{
			k \lambda
		}
	\right)^{1 / (p - 1)}
	\\
	&
	\le
	(r_1 - r_0)
	\left(
		\min
		\left\{
			r_1 - r_0,
			\frac{
				1
			}{
				k \lambda
			}
		\right\}
	\right)^{1 / (p - 1)};
\end{align*}
therefore,
\begin{align*}
	&
	\int_{r_0}^{r_1}
	dt
	\,
	\left(
		\frac{
			1
		}{
			t^{1+a}
		}
		\int_{r_0}^t
		\xi^{1+a}
		e^{
			- k
			\int_\xi^t
			b (\zeta)
			\,
			d\zeta
		}
		f (\xi, \beta M (\xi; u))
		\,
		d\xi
	\right)^{1 / (p - 1)}
	\\
	&
	\quad
	\le
	\gamma_5
	\min
	\left\{
		(r_1 - r_0)^{p / (p-1)},
		\frac{
			r_1 - r_0
		}{
			\lambda^{1 / (p-1)}
		}
	\right\}
	f^{1 / (p-1)} (\xi_*, \beta M (\xi_*; u)),
\end{align*}
where the constant $\gamma_5 > 0$ depends only on $p$ and $k$,
whence estimate~\eqref{L3.2.1} immediately follows according to Lemma~\ref{L3.1}.

Now, let $\sigma r_0 < r_1$ and $N$ be the maximal integer such that 
$\sigma^N r_0 < r_1$. We put $\rho_i = \sigma^i r_0$, $i = 0, \ldots, N$, 
and $\rho_{N+1} = r_1$. 
It can be seen that
\begin{align}
	&
	\int_{
		\rho_{i-1}
	}^{
		\rho_i
	}
	dt
	\,
	\left(
		\frac{
			1
		}{
			t^{1+a}
		}
		\int_{r_0}^t
		\xi^{1+a}
		e^{
			- k
			\int_\xi^t
			b (\zeta)
			\,
			d\zeta
		}
		f (\xi, \beta M (\xi; u))
		\,
		d\xi
	\right)^{1 / (p - 1)}
	\nonumber
	\\
	&
	\quad
	\le
	2^{1 / (p - 1)}
	\int_{
		\rho_{i-1}
	}^{
		\rho_i
	}
	dt
	\,
	\left(
		\frac{
			1
		}{
			t^{1+a}
		}
		\int_{
			\rho_{i-1}
		}^t
		\xi^{1+a}
		e^{
			- k
			\int_\xi^t
			b (\zeta)
			\,
			d\zeta
		}
		f (\xi, \beta M (\xi; u))
		\,
		d\xi
	\right)^{1 / (p - 1)}
	\nonumber
	\\
	&
	\quad
	\phantom{{}\le}
	+
	2^{1 / (p - 1)}
	\int_{
		\rho_{i-1}
	}^{
		\rho_i
	}
	dt
	\,
	\left(
		\frac{
			1
		}{
			t^{1+a}
		}
		\int_{
			r_0
		}^{
			\rho_{i-1}
		}
		\xi^{1+a}
		e^{
			- k
			\int_\xi^t
			b (\zeta)
			\,
			d\zeta
		}
		f (\xi, \beta M (\xi; u))
		\,
		d\xi
	\right)^{1 / (p - 1)}
	\label{PL3.2.2}
\end{align}
for all $i = 2, \ldots, N + 1$.
Repeating the previous arguments, we obtain
\begin{align}
	&
	M (\rho_i; u) - M (\rho_{i-1}; u)
	\nonumber
	\\
	&
	\quad
	\ge
	\gamma_6
	\int_{
		\rho_{i-1}
	}^{
		\rho_i
	}
	dt
	\,
	\left(
		\frac{
			1
		}{
			t^{1+a}
		}
		\int_{
			\rho_{i-1}
		}^t
		\xi^{1+a}
		e^{
			- k
			\int_\xi^t
			b (\zeta)
			\,
			d\zeta
		}
		f (\xi, \beta M (\xi; u))
		\,
		d\xi
	\right)^{1 / (p - 1)}
	\label{PL3.2.3}
\end{align}
for all $i = 2, \ldots, N + 1$, 
where the constant $\gamma_6 > 0$ depends only on
$n$, $p$, $k$, $C_1$, $C_2$, and $\beta$.
Analogously,
\begin{align}
	&
	M (\rho_1; u) - M (r_0; u)
	\nonumber
	\\
	&
	\quad
	\ge
	\gamma_6
	\int_{
		r_0
	}^{
		\rho_1
	}
	dt
	\,
	\left(
		\frac{
			1
		}{
			t^{1+a}
		}
		\int_{
			r_0
		}^t
		\xi^{1+a}
		e^{
			- k
			\int_\xi^t
			b (\zeta)
			\,
			d\zeta
		}
		f (\xi, \beta M (\xi; u))
		\,
		d\xi
	\right)^{1 / (p - 1)}.
	\label{PL3.2.4}
\end{align}

Further, in the case of $p \ge 2$, we have
\begin{align*}
	&
	\left(
		\int_{
			r_0
		}^{
			\rho_{i-1}
		}
		\xi^{1+a}
		e^{
			- k
			\int_\xi^t
			b (\zeta)
			\,
			d\zeta
		}
		f (\xi, \beta M (\xi; u))
		\,
		d\xi
	\right)^{1 / (p - 1)}
	\\
	&
	\quad
	\le
	\sum_{j=2}^i
	\left(
		\int_{
			\rho_{j-2}
		}^{
			\rho_{j-1}
		}
		\xi^{1+a}
		e^{
			- k
			\int_\xi^t
			b (\zeta)
			\,
			d\zeta
		}
		f (\xi, \beta M (\xi; u))
		\,
		d\xi
	\right)^{1 / (p - 1)}
	\\
	&
	\quad
	\le
	\sum_{j=2}^i
	\left(
		\int_{
			\rho_{j-2}
		}^{
			\rho_{j-1}
		}
		\xi^{1+a}
		e^{
			- k
			\int_\xi^{
				\rho_{j-1}
			}
			b (\zeta)
			\,
			d\zeta
		}
		f (\xi, \beta M (\xi; u))
		\,
		d\xi
	\right)^{1 / (p - 1)}
\end{align*}
for all $t \in (\rho_{i-1}, \rho_i)$, $i = 2, \ldots, N + 1$.
In particular,
\begin{align*}
	&
	\sum_{i=2}^{N+1}
	\int_{
		\rho_{i-1}
	}^{
		\rho_i
	}
	dt
	\,
	\left(
		\frac{
			1
		}{
			t^{1+a}
		}
		\int_{
			r_0
		}^{
			\rho_{i-1}
		}
		\xi^{1+a}
		e^{
			- k
			\int_\xi^t
			b (\zeta)
			\,
			d\zeta
		}
		f (\xi, \beta M (\xi; u))
		\,
		d\xi
	\right)^{1 / (p - 1)}
	\\
	&
	\quad
	\le
	\sum_{i=2}^{N+1}
	\int_{
		\rho_{i-1}
	}^{
		\rho_i
	}
	\frac{
		dt
	}{
		\rho_{i-1}^{(1 + a) / (p - 1)}
	}
	\left(
		\int_{
			r_0
		}^{
			\rho_{i-1}
		}
		\xi^{1+a}
		e^{
			- k
			\int_\xi^t
			b (\zeta)
			\,
			d\zeta
		}
		f (\xi, \beta M (\xi; u))
		\,
		d\xi
	\right)^{1 / (p - 1)}
	\\
	&
	\quad
	\le
	\sum_{i=2}^{N+1}
	\frac{
		\rho_i - \rho_{i-1}
	}{
		\rho_{i-1}^{(1 + a) / (p - 1)}
	}
	\sum_{j=2}^i
	\left(
		\int_{
			\rho_{j-2}
		}^{
			\rho_{j-1}
		}
		\xi^{1+a}
		e^{
			- k
			\int_\xi^{
				\rho_{j-1}
			}
			b (\zeta)
			\,
			d\zeta
		}
		f (\xi, \beta M (\xi; u))
		\,
		d\xi
	\right)^{1 / (p - 1)},
\end{align*}
whence in accordance with the evident inequalities
\begin{align*}
	&
	\int_{
		\rho_{j-2}
	}^{
		\rho_{j-1}
	}
	\xi^{1+a}
	e^{
		- k
		\int_\xi^{
			\rho_{j-1}
		}
		b (\zeta)
		\,
		d\zeta
	}
	f (\xi, \beta M (\xi; u))
	\,
	d\xi
	\\
	&
	\quad
	\le
	\rho_{j-1}^{1+a}
	\int_{
		\rho_{j-2}
	}^{
		\rho_{j-1}
	}
	e^{
		- k
		\int_\xi^{
			\rho_{j-1}
		}
		b (\zeta)
		\,
		d\zeta
	}
	f (\xi, \beta M (\xi; u))
	\,
	d\xi
\end{align*}
and
$\rho_i - \rho_{i-1} \le \sigma \rho_{i-1}$,
$2 \le j \le i \le N + 1$,
we obtain 
\begin{align}
	&
	\sum_{i=2}^{N+1}
	\int_{
		\rho_{i-1}
	}^{
		\rho_i
	}
	dt
	\,
	\left(
		\frac{
			1
		}{
			t^{1+a}
		}
		\int_{
			r_0
		}^{
			\rho_{i-1}
		}
		\xi^{1+a}
		e^{
			- k
			\int_\xi^t
			b (\zeta)
			\,
			d\zeta
		}
		f (\xi, \beta M (\xi; u))
		\,
		d\xi
	\right)^{1 / (p - 1)}
	\nonumber
	\\
	&
	\quad
	\le
	\sigma
	\sum_{i=2}^{N+1}
	\sum_{j=2}^{i}
	\left(
		\frac{
			\rho_{j-1}
		}{
			\rho_{i-1}
		}
	\right)^{(a - p + 2) / (p - 1)}
	\nonumber
	\\
	&
	\quad
	\phantom{{}\le}
	\times
	\rho_{j-1}
	\left(
		\int_{
			\rho_{j-2}
		}^{
			\rho_{j-1}
		}
		e^{
			- k
			\int_\xi^{
				\rho_{j-1}
			}
			b (\zeta)
			\,
			d\zeta
		}
		f (\xi, \beta M (\xi; u))
		\,
		d\xi
	\right)^{1 / (p - 1)}
	\nonumber
	\\
	&
	\quad
	=
	\sigma
	\sum_{j=2}^{N+1}
	\sum_{i=j}^{N+1}
	\left(
		\frac{
			\rho_{j-1}
		}{
			\rho_{i-1}
		}
	\right)^{(a - p + 2) / (p - 1)}
	\nonumber
	\\
	&
	\quad
	\phantom{{}=}
	\times
	\rho_{j-1}
	\left(
		\int_{
			\rho_{j-2}
		}^{
			\rho_{j-1}
		}
		e^{
			- k
			\int_\xi^{
				\rho_{j-1}
			}
			b (\zeta)
			\,
			d\zeta
		}
		f (\xi, \beta M (\xi; u))
		\,
		d\xi
	\right)^{1 / (p - 1)}.
	\label{PL3.2.5}
\end{align}
From the definition of the real numbers $\rho_i$, $i = 0, \ldots, N + 1$, it follows that
$$
	\sigma
	\sum_{i=j}^{N+1}
	\left(
		\frac{
			\rho_{j-1}
		}{
			\rho_{i-1}
		}
	\right)^{(a - p + 2) / (p - 1)}
	=
	\sum_{i=j}^{N+1}
	\sigma^{- (a - p + 2) (i - j) / (p - 1) + 1}
	\le
	\gamma_7
$$
for all $j = 2, \ldots, N + 1$, where the constant $\gamma_7 > 0$ 
depends only on $p$, $a$, and $\sigma$.
Consequently, relation~\eqref{PL3.2.5} implies the estimate
\begin{align}
	&
	\sum_{i=2}^{N+1}
	\int_{
		\rho_{i-1}
	}^{
		\rho_i
	}
	dt
	\,
	\left(
		\frac{
			1
		}{
			t^{1+a}
		}
		\int_{
			r_0
		}^{
			\rho_{i-1}
		}
		\xi^{1+a}
		e^{
			- k
			\int_\xi^t
			b (\zeta)
			\,
			d\zeta
		}
		f (\xi, \beta M (\xi; u))
		\,
		d\xi
	\right)^{1 / (p - 1)}
	\nonumber
	\\
	&
	\quad
	\le
	\gamma_7
	\sum_{j=2}^{N+1}
	\rho_{j-1}
	\left(
		\int_{
			\rho_{j-2}
		}^{
			\rho_{j-1}
		}
		e^{
			- k
			\int_\xi^{
				\rho_{j-1}
			}
			b (\zeta)
			\,
			d\zeta
		}
		f (\xi, \beta M (\xi; u))
		\,
		d\xi
	\right)^{1 / (p - 1)}.
	\label{PL3.2.6}
\end{align}
In so doing, Corollary~\ref{L3.1} enable us to assert that
\begin{align*}
	&
	M (\rho_{j-1}; u) - M (\rho_{j-2}; u)
	\\
	&
	\quad
	\ge
	\gamma_2
	(\rho_{j-1} - \rho_{j-2})
	\left(
		\int_{
			\rho_{j-2}
		}^{
			\rho_{j-1}
		}
		e^{
			- k
			\int_\xi^{
				\rho_{j-1}
			}
			b (\zeta)
			\,
			d\zeta
		}
		f (\xi, \beta M (\xi; u))
		\,
		d\xi
	\right)^{1 / (p - 1)}
	\\
	&
	\quad
	=
	\gamma_2
	\left(
		1 - \frac{1}{\sigma}
	\right)
	\rho_{j-1}
	\left(
		\int_{
			\rho_{j-2}
		}^{
			\rho_{j-1}
		}
		e^{
			- k
			\int_\xi^{
				\rho_{j-1}
			}
			b (\zeta)
			\,
			d\zeta
		}
		f (\xi, \beta M (\xi; u))
		\,
		d\xi
	\right)^{1 / (p - 1)}
\end{align*}
for all $j = 2, \ldots, N + 1$.
Thus,
\begin{align}
	&
	M (\rho_N; u) - M (r_0; u)
	\nonumber
	\\
	&
	\quad
	\ge
	\gamma_8
	\sum_{i=2}^{N+1}
	\int_{
		\rho_{i-1}
	}^{
		\rho_i
	}
	dt
	\,
	\left(
		\frac{
			1
		}{
			t^{1+a}
		}
		\int_{
			r_0
		}^{
			\rho_{i-1}
		}
		\xi^{1+a}
		e^{
			- k
			\int_\xi^t
			b (\zeta)
			\,
			d\zeta
		}
		f (\xi, \beta M (\xi; u))
		\,
		d\xi
	\right)^{1 / (p - 1)},
	\label{PL3.2.7}
\end{align}
where the constant $\gamma_8 > 0$ depends only on
$n$, $p$, $a$, $k$, $\sigma$, $C_1$, $C_2$, and $\beta$.

Now, assume that $1 < p < 2$. 
Since $a > p - 2$, there exists a real number $\delta > 0$ satisfying the condition 
$a - p + 2 - \delta > 0$.
In particular, we have $1 + a - \delta > 0$.
It is obvious that
\begin{align*}
	&
	\int_{
		r_0
	}^{
		\rho_{i-1}
	}
	\xi^{1+a}
	e^{
		- k
		\int_\xi^t
		b (\zeta)
		\,
		d\zeta
	}
	f (\xi, \beta M (\xi; u))
	\,
	d\xi
	\\
	&
	\quad
	=
	\sum_{j=2}^i
	\int_{
		\rho_{j-2}
	}^{
		\rho_{j-1}
	}
	\xi^{1+a}
	e^{
		- k
		\int_\xi^t
		b (\zeta)
		\,
		d\zeta
	}
	f (\xi, \beta M (\xi; u))
	\,
	d\xi
	\\
	&
	\quad
	\le
	\sum_{j=2}^i
	\int_{
		\rho_{j-2}
	}^{
		\rho_{j-1}
	}
	\xi^{1+a}
	e^{
		- k
		\int_\xi^{
			\rho_{j-1}
		}
		b (\zeta)
		\,
		d\zeta
	}
	f (\xi, \beta M (\xi; u))
	\,
	d\xi
\end{align*}
for all $t \in (\rho_{i-1}, \rho_i)$, $i = 2, \ldots, N + 1$.
Combining this with the estimates
\begin{align*}
	&
	\int_{
		\rho_{j-2}
	}^{
		\rho_{j-1}
	}
	\xi^{1+a}
	e^{
		- k
		\int_\xi^{
			\rho_{j-1}
		}
		b (\zeta)
		\,
		d\zeta
	}
	f (\xi, \beta M (\xi; u))
	\,
	d\xi
	\\
	&
	\quad
	\le
	\rho_{j-1}^\delta
	\int_{
		\rho_{j-2}
	}^{
		\rho_{j-1}
	}
	\xi^{1 + a - \delta}
	e^{
		- k
		\int_\xi^{
			\rho_{j-1}
		}
		b (\zeta)
		\,
		d\zeta
	}
	f (\xi, \beta M (\xi; u))
	\,
	d\xi,
	\:
	j = 2, \ldots, N + 1,
\end{align*}
we obtain
\begin{align*}
	&
	\int_{
		r_0
	}^{
		\rho_{i-1}
	}
	\xi^{1+a}
	e^{
		- k
		\int_\xi^t
		b (\zeta)
		\,
		d\zeta
	}
	f (\xi, \beta M (\xi; u))
	\,
	d\xi
	\\
	&
	\quad
	\le
	\sum_{j=2}^i
	\rho_{j-1}^\delta
	\int_{
		\rho_{j-2}
	}^{
		\rho_{j-1}
	}
	\xi^{1 + a - \delta}
	e^{
		- k
		\int_\xi^{
			\rho_{j-1}
		}
		b (\zeta)
		\,
		d\zeta
	}
	f (\xi, \beta M (\xi; u))
	\,
	d\xi
\end{align*}
for all $t \in (\rho_{i-1}, \rho_i)$, $i = 2, \ldots, N + 1$,
whence in accordance with the H\"older inequality it follows that
\begin{align*}
	&
	\int_{
		r_0
	}^{
		\rho_{i-1}
	}
	\xi^{1+a}
	e^{
		- k
		\int_\xi^t
		b (\zeta)
		\,
		d\zeta
	}
	f (\xi, \beta M (\xi; u))
	\,
	d\xi
	\le
	\left(
		\sum_{j=2}^i
		\rho_{j-1}^{\delta / (2 - p)}
	\right)^{2 - p}
	\\
	&
	\quad
	\times
	\left(
		\sum_{j=2}^i
		\left(
			\int_{
				\rho_{j-2}
			}^{
				\rho_{j-1}
			}
			\xi^{1 + a - \delta}
			e^{
				- k
				\int_\xi^{
					\rho_{j-1}
				}
				b (\zeta)
				\,
				d\zeta
			}
			f (\xi, \beta M (\xi; u))
			\,
			d\xi
		\right)^{1 / (p - 1)}
	\right)^{p - 1}
\end{align*}
for all $t \in (\rho_{i-1}, \rho_i)$, $i = 2, \ldots, N + 1$.
At the same time,
$$
	\sum_{j=2}^i
	\rho_{j-1}^{\delta / (2 - p)}
	=
	\rho_{i-1}^{\delta / (2 - p)}
	\sum_{j=2}^i
	\sigma^{- \delta (i - j) / (2 - p)}
	\le
	\frac{
		\rho_{i-1}^{\delta / (2 - p)}
	}{
		1 - \sigma^{- \delta / (2 - p)}
	},
	\:
	i = 2, \ldots, N + 1.
$$
Consequently, we have
\begin{align*}
	&
	\left(
		\int_{
			r_0
		}^{
			\rho_{i-1}
		}
		\xi^{1+a}
		e^{
			- k
			\int_\xi^t
			b (\zeta)
			\,
			d\zeta
		}
		f (\xi, \beta M (\xi; u))
		\,
		d\xi
	\right)^{1 / (p - 1)}
	\\
	&
	\quad
	\le
	\gamma_9
	\rho_{i-1}^{\delta / (p - 1)}
	\sum_{j=2}^i
	\left(
		\int_{
			\rho_{j-2}
		}^{
			\rho_{j-1}
		}
		\xi^{1 + a - \delta}
		e^{
			- k
			\int_\xi^{
				\rho_{j-1}
			}
			b (\zeta)
			\,
			d\zeta
		}
		f (\xi, \beta M (\xi; u))
		\,
		d\xi
	\right)^{1 / (p - 1)}
\end{align*}
for all $t \in (\rho_{i-1}, \rho_i)$, $i = 2, \ldots, N + 1$, 
where the constant $\gamma_9 > 0$ depends only on
$\delta$, $p$, $a$, and $\sigma$.
This immediately implies the estimate
\begin{align*}
	&
	\sum_{i=2}^{N+1}
	\int_{
		\rho_{i-1}
	}^{
		\rho_i
	}
	dt
	\,
	\left(
		\frac{
			1
		}{
			t^{1+a}
		}
		\int_{
			r_0
		}^{
			\rho_{i-1}
		}
		\xi^{1+a}
		e^{
			- k
			\int_\xi^t
			b (\zeta)
			\,
			d\zeta
		}
		f (\xi, \beta M (\xi; u))
		\,
		d\xi
	\right)^{1 / (p - 1)}
	\\
	&
	\quad
	\le
	\sum_{i=2}^{N+1}
	\int_{
		\rho_{i-1}
	}^{
		\rho_i
	}
	\frac{
		dt
	}{
		\rho_{i-1}^{(1 + a) / (p - 1)}
	}
	\left(
		\int_{
			r_0
		}^{
			\rho_{i-1}
		}
		\xi^{1+a}
		e^{
			- k
			\int_\xi^t
			b (\zeta)
			\,
			d\zeta
		}
		f (\xi, \beta M (\xi; u))
		\,
		d\xi
	\right)^{1 / (p - 1)}
	\\
	&
	\quad
	\le
	\gamma_9
	\sum_{i=2}^{N+1}
	\frac{
		\rho_i - \rho_{i-1}
	}{
		\rho_{i-1}^{(1 + a - \delta) / (p - 1)}
	}
	\\
	&
	\quad
	\phantom{{}\le}
	\times
	\sum_{j=2}^i
	\left(
		\int_{
			\rho_{j-2}
		}^{
			\rho_{j-1}
		}
		\xi^{1 + a - \delta}
		e^{
			- k
			\int_\xi^{
				\rho_{j-1}
			}
			b (\zeta)
			\,
			d\zeta
		}
		f (\xi, \beta M (\xi; u))
		\,
		d\xi
	\right)^{1 / (p - 1)}.
\end{align*}
Therefore, taking into account the fact that
\begin{align*}
	&
	\int_{
		\rho_{j-2}
	}^{
		\rho_{j-1}
	}
	\xi^{1 + a - \delta}
	e^{
		- k
		\int_\xi^{
			\rho_{j-1}
		}
		b (\zeta)
		\,
		d\zeta
	}
	f (\xi, \beta M (\xi; u))
	\,
	d\xi
	\\
	&
	\quad
	\le
	\rho_{j-1}^{1 + a - \delta}
	\int_{
		\rho_{j-2}
	}^{
		\rho_{j-1}
	}
	e^{
		- k
		\int_\xi^{
			\rho_{j-1}
		}
		b (\zeta)
		\,
		d\zeta
	}
	f (\xi, \beta M (\xi; u))
	\,
	d\xi
\end{align*}
and
$\rho_i - \rho_{i-1} \le \sigma \rho_{i-1}$,
$2 \le j \le i \le N + 1$,
we obtain 
\begin{align*}
	&
	\sum_{i=2}^{N+1}
	\int_{
		\rho_{i-1}
	}^{
		\rho_i
	}
	dt
	\,
	\left(
		\frac{
			1
		}{
			t^{1+a}
		}
		\int_{
			r_0
		}^{
			\rho_{i-1}
		}
		\xi^{1+a}
		e^{
			- k
			\int_\xi^t
			b (\zeta)
			\,
			d\zeta
		}
		f (\xi, \beta M (\xi; u))
		\,
		d\xi
	\right)^{1 / (p - 1)}
	\\
	&
	\quad
	\le
	\gamma_9
	\sigma
	\sum_{i=2}^{N+1}
	\sum_{j=2}^i
	\left(
		\frac{
			\rho_{j-1}
		}{
			\rho_{i-1}
		}
	\right)^{(a - p + 2 - \delta) / (p - 1)}
	\\
	&
	\quad
	\phantom{{}\le}
	\times
	\rho_{j-1}
	\left(
		\int_{
			\rho_{j-2}
		}^{
			\rho_{j-1}
		}
		e^{
			- k
			\int_\xi^{
				\rho_{j-1}
			}
			b (\zeta)
			\,
			d\zeta
		}
		f (\xi, \beta M (\xi; u))
		\,
		d\xi
	\right)^{1 / (p - 1)}
	\\
	&
	\quad
	=
	\gamma_9
	\sigma
	\sum_{j=2}^{N+1}
	\sum_{i=j}^{N+1}
	\left(
		\frac{
			\rho_{j-1}
		}{
			\rho_{i-1}
		}
	\right)^{(a - p + 2 - \delta) / (p - 1)}
	\\
	&
	\quad
	\phantom{{}=}
	\times
	\rho_{j-1}
	\left(
		\int_{
			\rho_{j-2}
		}^{
			\rho_{j-1}
		}
		e^{
			- k
			\int_\xi^{
				\rho_{j-1}
			}
			b (\zeta)
			\,
			d\zeta
		}
		f (\xi, \beta M (\xi; u))
		\,
		d\xi
	\right)^{1 / (p - 1)}.
\end{align*}
Since
$$
	\sum_{i=j}^{N+1}
	\left(
		\frac{
			\rho_{j-1}
		}{
			\rho_{i-1}
		}
	\right)^{(a - p + 2 - \delta) / (p - 1)}
	=
	\sum_{i=j}^{N+1}
	\sigma^{- (a - p + 2 - \delta) (i - j) / (p - 1)}
	\le
	\gamma_{10}
$$
for all $j = 2, \ldots, N + 1$, where the constant $\gamma_{10} > 0$ 
depends only on $\delta$, $p$, $a$, and $\sigma$, 
this again implies inequality~\eqref{PL3.2.6},
whence we immediately derive~\eqref{PL3.2.7}.

From~\eqref{PL3.2.2}, \eqref{PL3.2.3}, and~\eqref{PL3.2.7},
it follows that
\begin{align*}
	&
	M (r_1; u) - M (r_0; u)
	\\
	&
	\quad
	\ge
	\gamma_{11}
	\sum_{i=2}^{N+1}
	\int_{
		\rho_{i-1}
	}^{
		\rho_i
	}
	dt
	\,
	\left(
		\frac{
			1
		}{
			t^{1+a}
		}
		\int_{r_0}^t
		\xi^{1+a}
		e^{
			- k
			\int_\xi^t
			b (\zeta)
			\,
			d\zeta
		}
		f (\xi, \beta M (\xi; u))
		\,
		d\xi
	\right)^{1 / (p - 1)}
	\\
	&
	\quad
	=
	\gamma_{11}
	\int_{
		\rho_1
	}^{
		r_1
	}
	dt
	\,
	\left(
		\frac{
			1
		}{
			t^{1+a}
		}
		\int_{r_0}^t
		\xi^{1+a}
		e^{
			- k
			\int_\xi^t
			b (\zeta)
			\,
			d\zeta
		}
		f (\xi, \beta M (\xi; u))
		\,
		d\xi
	\right)^{1 / (p - 1)},
\end{align*}
where the constant $\gamma_{11} > 0$ depends only on
$n$, $p$, $a$, $k$, $\sigma$, $C_1$, $C_2$, and $\beta$.
Thus, to complete the proof it remains to combine the last formula with~\eqref{PL3.2.4}.
\end{proof}

\begin{Lemma}\label{L3.3}
In the hypotheses of Lemma~$\ref{L3.2}$, let $\sigma^{1/2} r_0 \le r_1$. Then
\begin{align*}
	&
	M (r_1; u) - M (r_0; u)
	\\
	&
	\quad
	\ge
	\gamma_{12}
	r_1^{
		- (a - p + 2) / (p - 1)
	}
	\left(
		\int_{r_0}^{r_1}
		\xi^{1+a}
		e^{
			- k
			\int_\xi^{r_1}
			b (\zeta)
			\,
			d\zeta
		}
		f (\xi, \beta M (\xi; u))
		\,
		d\xi
	\right)^{1 / (p - 1)},
\end{align*}
where the constant $\gamma_{12} > 0$ depends only on
$n$, $p$, $a$, $k$, $\sigma$, $C_1$, $C_2$, and $\beta$.
\end{Lemma}

\begin{proof}
Using Corollary~\ref{C3.1}, one can show that
\begin{align*}
	&
	M (r_1; u) - M (\sigma^{- 1 / 2} r_1; u)
	\\
	&
	\quad
	\ge
	\gamma_2
	(1 - \sigma^{- 1 / 2})
	r_1
	\left(
		\int_{
			\sigma^{- 1 / 2} r_1
		}^{
			r_1
		}
		e^{
			- k
			\int_\xi^{r_1}
			b (\zeta)
			\,
			d\zeta
		}
		f (\xi, \beta M (\xi; u))
		\,
		d\xi
	\right)^{1 / (p - 1)}
	\\
	&
	\quad
	\ge
	\gamma_2
	(1 - \sigma^{- 1 / 2})
	r_1^{
		- (a - p + 2) / (p - 1)
	}
	\\
	&
	\quad
	\phantom{{}\ge}
	\times
	\left(
		\int_{
			\sigma^{- 1 / 2} r_1
		}^{
			r_1
		}
		\xi^{1+a}
		e^{
			- k
			\int_\xi^{r_1}
			b (\zeta)
			\,
			d\zeta
		}
		f (\xi, \beta M (\xi; u))
		\,
		d\xi
	\right)^{1 / (p - 1)}.
\end{align*}
Combining this with the inequality
\begin{align*}
	&
	M (r_1; u) - M (r_0; u)
	\\
	&
	\quad
	\ge
	\gamma_4
	\int_{
		\sigma^{- 1 / 2} r_1
	}^{
		r_1
	}
	dt
	\,
	\left(
		\frac{
			1
		}{
			t^{1+a}
		}
		\int_{
			r_0
		}^{
			\sigma^{- 1 / 2} r_1
		}
		\xi^{1+a}
		e^{
			- k
			\int_\xi^{r_1}
			b (\zeta)
			\,
			d\zeta
		}
		f (\xi, \beta M (\xi; u))
		\,
		d\xi
	\right)^{1 / (p - 1)}
	\\
	&
	\quad
	\ge
	\gamma_4
	(1 - \sigma^{-1/2})
	r_1^{- (a - p + 2) / (p - 1)}
	\\
	&
	\quad
	\phantom{{}\ge}
	\times
	\left(
		\int_{
			r_0
		}^{
			\sigma^{- 1 / 2} r_1
		}
		\xi^{1+a}
		e^{
			- k
			\int_\xi^{r_1}
			b (\zeta)
			\,
			d\zeta
		}
		f (\xi, \beta M (\xi; u))
		\,
		d\xi
	\right)^{1 / (p - 1)},
\end{align*}
which follows from Lemma~\ref{L3.2}, we complete the proof.
\end{proof}

Further in this section, we assume that
$$
	\beta
	=
	\left(
		\min
		\left\{
			\frac{
				1
			}{
				4^{p / (p - 1) + 2}
				\sigma^{1 / 2}
			},
			\frac{
				(1 - \sigma^{- 1 / 2})
				(a - p + 2)
			}{
				8^{p / (p - 1) + 1}
				(p - 1)
			}
		\right\}
	\right)^2
$$
and
$$
	\alpha
	=
	\left(
		\min
		\left\{
			\gamma_2
			\beta^{1 / 2},
			\frac{
				\gamma_2
			}{
				4^{p / (p - 1) + 1}
				\sigma^{1 / 2}
			},
			\frac{
				\gamma_3
			}{
				4^{p / (p - 1)}
			},
			\frac{
				\gamma_{12}
				(a - p + 2)
			}{
				4^{p / (p - 1)}
				(p - 1)
			},
			\frac{
				\gamma_4
			}{
			2^{p / (p - 1)}
			}
		\right\}
	\right)^{p - 1}.
$$

\begin{Lemma}\label{L3.4}
Suppose that
$M (r_0; u) \le \beta^{1/2} M (r_1; u) \le M (r_0 + 0; u)$
and
\begin{equation}
	M (r_0; u)
	\ge
	\int_{R_0}^{r_0}
	dt
	\,
	\left(
		\frac{\alpha}{t^{1+a}}
		\int_{R_0}^t
		\xi^{1+a}
		e^{
			- k
			\int_\xi^t
			b (\zeta)
			\,
			d\zeta
		}
		f (\xi, \beta M (\xi; u))
		\,
		d\xi
	\right)^{1 / (p-1)}
	\label{L3.4.1}
\end{equation}
for some real numbers
$R_0 < r_0 < r_1 < R_1$.
If
$\sigma^{1/2} r_0 \le r_1$,
then
\begin{align}
	&
	M (r_1; u) - M (r_0; u)
	\nonumber
	\\
	&
	\quad
	\ge
	2^{- p / (p - 1)}
	(1 - \sigma^{- 1/2})
	\beta^{- 1/2}
	r_0^{
		- (a - p + 2) / (p -1)
	}
	\nonumber
	\\
	&
	\quad
	\phantom{{}\ge}
	\times
	\left(
		\alpha
		\int_{R_0}^{r_0}
		\xi^{1+a}
		e^{
			- k
			\int_\xi^{r_0}
			b (\zeta)
			\,
			d\zeta
		}
		f (\xi, \beta M (\xi; u))
		\,
		d\xi
	\right)^{1 / (p-1)}.
	\label{L3.4.2}
\end{align}
\end{Lemma}

\begin{proof}
At first, let
\begin{align}
	&
	\frac{1}{2}
	\int_{R_0}^{r_0}
	\xi^{1+a}
	e^{
		- k
		\int_\xi^{r_0}
		b (\zeta)
		\,
		d\zeta
	}
	f (\xi, \beta M (\xi; u))
	\,
	d\xi
	\nonumber
	\\
	&
	\quad
	\le
	\int_{r_*}^{r_0}
	\xi^{1+a}
	e^{
		- k
		\int_\xi^{r_0}
		b (\zeta)
		\,
		d\zeta
	}
	f (\xi, \beta M (\xi; u))
	\,
	d\xi,
	\label{PL3.4.1}
\end{align}
where $r_* = \max \{ R_0, r_0 / \sigma^{1/2} \}$.
By Corollary~\ref{C3.1}, we obtain
\begin{align*}
	&
	M (\sigma^{1/2} r_0; u) - M (r_0 + 0; u)
	\\
	&
	\quad
	\ge
	\gamma_2
	(\sigma^{1/2} - 1)
	r_0
	\left(
		\int_{r_*}^{r_0}
		e^{
			- k
			\int_\xi^{r_0}
			b (\zeta)
			\,
			d\zeta
		}
		f (\xi, \beta M (\xi; u))
		\,
		d\xi
	\right)^{1 / (p - 1)}
	\\
	&
	\quad
	\ge
	\gamma_2
	2^{- 1 / (p - 1)}
	(\sigma^{1/2} - 1)
	r_0^{
		- (a - p + 2) / (p -1)
	}
	\\
	&
	\quad
	\phantom{{}\ge}
	\times
	\left(
		\int_{R_0}^{r_0}
		\xi^{1+a}
		e^{
			- k
			\int_\xi^{r_0}
			b (\zeta)
			\,
			d\zeta
		}
		f (\xi, \beta M (\xi; u))
		\,
		d\xi
	\right)^{1 / (p - 1)}.
\end{align*}
The last relation immediately implies~\eqref{L3.4.2}.

Now, assume that~\eqref{PL3.4.1} is not valid. In this case, we have
$r_* = r_0 / \sigma^{1/2} > R_0$
and
\begin{align}
	&
	\frac{1}{2}
	\int_{R_0}^{r_0}
	\xi^{1+a}
	e^{
		- k
		\int_\xi^{r_0}
		b (\zeta)
		\,
		d\zeta
	}
	f (\xi, \beta M (\xi; u))
	\,
	d\xi
	\nonumber
	\\
	&
	\quad
	\le
	\int_{R_0}^{r_*}
	\xi^{1+a}
	e^{
		- k
		\int_\xi^{r_0}
		b (\zeta)
		\,
		d\zeta
	}
	f (\xi, \beta M (\xi; u))
	\,
	d\xi.
	\label{PL3.4.2}
\end{align}
From~\eqref{L3.4.1}, it can be seen that
\begin{align*}
	M (r_0; u)
	&
	{}
	\ge
	\int_{r_*}^{r_0}
	\frac{
		dt
	}{
		r_0^{(1 + a) / (p - 1)}
	}
	\,
	\left(
		\alpha
		\int_{R_0}^{r_*}
		\xi^{1+a}
		e^{
			- k
			\int_\xi^{r_0}
			b (\zeta)
			\,
			d\zeta
		}
		f (\xi, \beta M (\xi; u))
		\,
		d\xi
	\right)^{1 / (p-1)}
	\\
	&
	=
	(1 - \sigma^{- 1/2})
	r_0^{
		- (a - p + 2) / (p -1)
	}
	\\
	&
	\phantom{{}=}
	\times
	\left(
		\alpha
		\int_{R_0}^{r_*}
		\xi^{1+a}
		e^{
			- k
			\int_\xi^{r_0}
			b (\zeta)
			\,
			d\zeta
		}
		f (\xi, \beta M (\xi; u))
		\,
		d\xi
	\right)^{1 / (p-1)}.
\end{align*}
Combining this with formula~\eqref{PL3.4.2} and the inequality
$
	M (r_1; u) - M (r_0; u) 
	\ge 
	(\beta^{- 1 / 2} - 1) M (r_0; u)
	\ge
	\beta^{- 1 / 2}
	M (r_0; u) / 2,
$
we again obtain~\eqref{L3.4.2}.
The proof is completed.
\end{proof}

\begin{Lemma}\label{L3.5}
Let
$R_0 < r_0 < r < R_1$,
$r \le \sigma^{1/2} r_0$
and, moreover,
$$
	M (\zeta; u)
	\ge
	\int_{R_0}^\zeta
	dt
	\,
	\left(
		\frac{\alpha}{t^{1+a}}
		\int_{R_0}^t
		\xi^{1+a}
		e^{
			- k
			\int_\xi^t
			b (\zeta)
			\,
			d\zeta
		}
		f (\xi, \beta M (\xi; u))
		\,
		d\xi
	\right)^{1 / (p-1)}
$$
for all $\zeta \in (R_0, r_0)$.
If
$M (r_0; u) \le \beta^{1/2} M (r; u) \le M (r_0 + 0; u)$,
then
\begin{align}
	&
	M (r; u) - M (r_0; u)
	\nonumber
	\\
	&
	\quad
	\ge
	\frac{
		2^{p / (p - 1)} (p - 1)
	}{
		a - p +2
	}
	(
		r_0^{
			- (a - p + 2) / (p -1)
		}
		-
		r^{
			- (a - p + 2) / (p -1)
		}
	)
	\nonumber
	\\
	&
	\quad
	\phantom{{}\ge}
	\times
	\left(
		\alpha
		\int_{R_0}^{r_0}
		\xi^{1+a}
		e^{
			- k
			\int_\xi^{r_0}
			b (\zeta)
			\,
			d\zeta
		}
		f (\xi, \beta M (\xi; u))
		\,
		d\xi
	\right)^{1 / (p-1)}.
	\label{L3.5.1}
\end{align}
\end{Lemma}

\begin{proof}
We put 
$
	r_1 
	=
	\max
	\{
		R_0,
		r_0 - \sigma^{-1/2} (r - r_0) / 2
	\}.
$
By Corollary~\ref{C3.1}, 
$$
	M (r; u) - M (r_0 + 0; u)
	\ge
	\gamma_2
	(r - r_0)
	\left(
		\int_{r_1}^{r_0}
		e^{
			- k
			\int_\xi^{r_0}
			b (\zeta)
			\,
			d\zeta
		}
		f (\xi, \beta M (\xi; u))
		\,
		d\xi
	\right)^{1 / (p - 1)}.
$$
Combining this with the inequality
\begin{equation}
	\frac{
		p - 1
	}{
		a - p +2
	}
	(
		r_0^{
			- (a - p + 2) / (p -1)
		}
		-
		r^{
			- (a - p + 2) / (p -1)
		}
	)
	\le
	r_0^{- (1 + a) / (p - 1)}
	(r - r_0),
	\label{PL3.5.2}
\end{equation}
we have
\begin{align}
	&
	M (r; u) - M (r_0 + 0; u)
	\nonumber
	\\
	&
	\quad
	\ge
	\frac{
		4^{p / (p - 1)}
		(p - 1)
	}{
		a - p +2
	}
	(
		r_0^{
			- (a - p + 2) / (p -1)
		}
		-
		r^{
			- (a - p + 2) / (p -1)
		}
	)
	\nonumber
	\\
	&
	\quad
	\phantom{{}\ge}
	\times
	\left(
		\alpha
		\int_{r_1}^{r_0}
		\xi^{1+a}
		e^{
			- k
			\int_\xi^{r_0}
			b (\zeta)
			\,
			d\zeta
		}
		f (\xi, \beta M (\xi; u))
		\,
		d\xi
	\right)^{1 / (p-1)}.
	\label{PL3.5.1}
\end{align}

The proof of Lemma~\ref{L3.5} is by induction over the positive integer $N$ 
defined as follows: $N = 1$ if $r_1 = R_0$; otherwise $N$ is the minimal 
positive integer such that 
${M (R_0 + 0; u)} \ge \beta^{N/2} {M (r_1; u)}$.

Consider the case of $N = 1$. If $r_1 - R_0 \le r_0 - r_1$, then
$r_0 \le \sigma^{1/2} R_0$. 
Hence, repeating the arguments given in the proof of~\eqref{PL3.5.1} 
with $r_1$ replaced by $R_0$, we obviously obtain~\eqref{L3.5.1}.
Let $r_1 - R_0 > r_0 - r_1$.
For $r_1 \le \sigma^{1/2} R_0$, taking into account
Corollary~\ref{C3.1}, we have
$$
	M (r_1; u) - M (R_0 + 0; u)
	\ge
	\gamma_2
	(r_1 - R_0)
	\left(
		\int_{R_0}^{r_1}
		e^{
			- k
			\int_\xi^{r_1}
			b (\zeta)
			\,
			d\zeta
		}
		f (\xi, \beta M (\xi; u))
		\,
		d\xi
	\right)^{1 / (p - 1)}.
$$
The last formula, bound~\eqref{PL3.5.2}, and the relation
$r_1 - R_0 > r_0 - r_1 = \sigma^{- 1 / 2} {(r - r_0) / 2}$
enable us to assert that
\begin{align*}
	&
	M (r_1; u) - M (R_0 + 0; u)
	\\
	&
	\quad
	\ge
	\frac{
		\gamma_2 
		(p - 1)
	}{
		2 \sigma^{1/2} (a - p +2)
	}
	(
		r_0^{
			- (a - p + 2) / (p -1)
		}
		-
		r^{
			- (a - p + 2) / (p -1)
		}
	)
	\nonumber
	\\
	&
	\quad
	\phantom{{}\ge}
	\times
	\left(
		\int_{R_0}^{r_1}
		\xi^{1+a}
		e^{
			- k
			\int_\xi^{r_1}
			b (\zeta)
			\,
			d\zeta
		}
		f (\xi, \beta M (\xi; u))
		\,
		d\xi
	\right)^{1 / (p-1)}.
\end{align*}
In so doing,
\begin{equation}
	M (r; u) - M (r_0; u) 
	\ge 
	(\beta^{- 1 / 2} - 1) M (r_0; u)
	\ge
	M (r_1; u);
	\label{PL3.5.12}
\end{equation}
therefore, we obtain
\begin{align*}
	&
	M (r; u) - M (r_0; u)
	\\
	&
	\quad
	\ge
	\frac{
		\gamma_2
		(p - 1)
	}{
		2
		\sigma^{1/2}
		(a - p +2)
	}
	(
		r_0^{
			- (a - p + 2) / (p -1)
		}
		-
		r^{
			- (a - p + 2) / (p -1)
		}
	)
	\\
	&
	\quad
	\phantom{{}\ge}
	\times
	\left(
		\int_{R_0}^{r_1}
		\xi^{1+a}
		e^{
			- k
			\int_\xi^{r_1}
			b (\zeta)
			\,
			d\zeta
		}
		f (\xi, \beta M (\xi; u))
		\,
		d\xi
	\right)^{1 / (p-1)},
\end{align*}
whence it can be seen that
\begin{align}
	&
	M (r; u) - M (r_0; u)
	\nonumber
	\\
	&
	\quad
	\ge
	\frac{
		4^{p / (p - 1)}
		(p - 1)
	}{
		a - p +2
	}
	(
		r_0^{
			- (a - p + 2) / (p -1)
		}
		-
		r^{
			- (a - p + 2) / (p -1)
		}
	)
	\nonumber
	\\
	&
	\quad
	\phantom{{}\ge}
	\times
	\left(
		\alpha
		\int_{R_0}^{r_1}
		\xi^{1+a}
		e^{
			- k
			\int_\xi^{r_0}
			b (\zeta)
			\,
			d\zeta
		}
		f (\xi, \beta M (\xi; u))
		\,
		d\xi
	\right)^{1 / (p-1)}.
	\label{PL3.5.3}
\end{align}

On the other hand, if of $r_1 > \sigma^{1/2} R_0$, 
then in accordance with Lemma~\ref{L3.3} we have
\begin{align*}
	&
	M (r_1; u) - M (R_0 + 0; u)
	\\
	&
	\quad
	\ge
	\gamma_{12}
	r_1^{
		- (a - p + 2) / (p - 1)
	}
	\left(
		\int_{R_0}^{r_1}
		\xi^{1+a}
		e^{
			- k
			\int_\xi^{r_1}
			b (\zeta)
			\,
			d\zeta
		}
		f (\xi, \beta M (\xi; u))
		\,
		d\xi
	\right)^{1 / (p - 1)}
	\\
	&
	\quad
	\ge
	\gamma_{12}
	(
		r_0^{
			- (a - p + 2) / (p - 1)
		}
		-
		r^{
			- (a - p + 2) / (p - 1)
		}
	)
	\\
	&
	\quad
	\phantom{{}\ge}
	\times
	\left(
		\int_{R_0}^{r_1}
		\xi^{1+a}
		e^{
			- k
			\int_\xi^{r_1}
			b (\zeta)
			\,
			d\zeta
		}
		f (\xi, \beta M (\xi; u))
		\,
		d\xi
	\right)^{1 / (p - 1)}.
\end{align*}
By~\eqref{PL3.5.12}, this implies the estimate
\begin{align*}
	&
	M (r; u) - M (r_0; u)
	\\
	&
	\quad
	\ge
	\gamma_{12}
	(
		r_0^{
			- (a - p + 2) / (p - 1)
		}
		-
		r^{
			- (a - p + 2) / (p - 1)
		}
	)
	\\
	&
	\quad
	\phantom{{}\ge}
	\times
	\left(
		\int_{R_0}^{r_1}
		\xi^{1+a}
		e^{
			- k
			\int_\xi^{r_1}
			b (\zeta)
			\,
			d\zeta
		}
		f (\xi, \beta M (\xi; u))
		\,
		d\xi
	\right)^{1 / (p - 1)},
\end{align*}
whence~\eqref{PL3.5.3} follows again.
Finally, summing~\eqref{PL3.5.1} and~\eqref{PL3.5.3}, we derive~\eqref{L3.5.1}.

Assume further that Lemma~\ref{L3.5} is proved for all $N \le N_0$,
where $N_0$ is  a positive integer. We shall prove the lemma
for $N = N_0 + 1$.

Let us construct the finite sequence of real numbers $R_0 = r_l < \ldots < r_2 < r_1$.
The real number $r_1$ is defined in the beginning of the proof.
If $r_i$ is already known, then we put
\begin{equation}
	r_{i+1}
	=
	\inf
	\{
		\xi \in (R_0, r_i)
		:
		M (\xi; u) > \beta^{1/2} M (r_i; u)
	\}.
	\label{PL3.5.4}
\end{equation}
In the case of $r_{i+1} = R_0$, we set $l = i + 1$ and stop.

From~(\ref{3.1}), it can be seen that
$
    \{
        \xi \in (R_0, r_i)
        :
        M (\xi; u) > \beta^{1/2} M (r_i; u)
    \}
    \ne
    \emptyset
$
for all $i = 1, \ldots, l - 1$.
Thus, the right-hand side of~(\ref{PL3.5.4}) is well-defined.
Also note that $l \ge 3$ as $N \ge 2$.

By $\Xi$ we mean the set of integers $\nu \in \{ 2, \ldots, l - 1 \}$
safisfying the conditions
$
	r_{i-1} \le \sigma^{1/2} r_i,
$
$
	r_{i-1} - r_i
	\le
	2^{-i + 1}
	(r_0 - r_1),
$
and
$$
	\int_{R_0}^{r_i}
	\xi^{1+a}
	e^{
		- k
		\int_\xi^{r_0}
		b (\zeta)
		\,
		d\zeta
	}
	f (\xi, \beta M (\xi; u))
	\,
	d\xi
	\ge
	\frac{1}{2}
	\int_{R_0}^{r_{i-1}}
	\xi^{1+a}
	e^{
		- k
		\int_\xi^{r_0}
		b (\zeta)
		\,
		d\zeta
	}
	f (\xi, \beta M (\xi; u))
	\,
	d\xi
$$
for all $i \in \{ 2, \ldots, \nu \}$.
We put $j = \max \Xi$ if the set $\Xi$ is not empty and $j = 1$, otherwise. 

As indicated above, to prove the lemma it is sufficient to establish the validity
of estimate~\eqref{PL3.5.3}.
It presents no special problems to verify that at least
one of the following propositions is valid:

\begin{enumerate}
\item [(1)]
$\sigma^{1/2} r_{j+1} < r_j$ and
\begin{align}
	&
	\int_{
		r_{j+1}
	}^{
		r_j
	}
	\xi^{1+a}
	e^{
		- k
		\int_\xi^{r_0}
		b (\zeta)
		\,
		d\zeta
	}
	f (\xi, \beta M (\xi; u))
	\,
	d\xi
	\nonumber
	\\
	&
	\quad
	\ge
	\frac{1}{2}
	\int_{R_0}^{r_j}
	\xi^{1+a}
	e^{
		- k
		\int_\xi^{r_0}
		b (\zeta)
		\,
		d\zeta
	}
	f (\xi, \beta M (\xi; u))
	\,
	d\xi;
	\label{PL3.5.5}
\end{align}

\item [(2)]
$\sigma^{1/2} r_{j+1} \ge r_j$
and
$r_j - r_{j+1} > 2^{-j} (r_0 - r_1)$;

\item [(3)]
$r_j - r_{j+1} \le 2^{-j} (r_0 - r_1)$
and, moreover, relation~(\ref{PL3.5.5}) holds;

\item [(4)]
$\sigma^{1/2} r_{j+1} < r_j$ and
$$
	\int_{
		R_0
	}^{
		r_{j+1}
	}
	\xi^{1+a}
	e^{
		- k
		\int_\xi^{r_0}
		b (\zeta)
		\,
		d\zeta
	}
	f (\xi, \beta M (\xi; u))
	\,
	d\xi
	\ge
	\frac{1}{2}
	\int_{R_0}^{r_j}
	\xi^{1+a}
	e^{
		- k
		\int_\xi^{r_0}
		b (\zeta)
		\,
		d\zeta
	}
	f (\xi, \beta M (\xi; u))
	\,
	d\xi.
$$
\end{enumerate}

In case (1), Lemma~\ref{L3.3} implies the estimate
\begin{align*}
	&
	M (r_j; u) - M (r_{j+1} + 0; u)
	\\
	&
	\quad
	\ge
	\gamma_{12}
	r_j^{
		- (a - p + 2) / (p - 1)
	}
	\left(
		\int_{r_{j+1}}^{r_j}
		\xi^{1+a}
		e^{
			- k
			\int_\xi^{r_j}
			b (\zeta)
			\,
			d\zeta
		}
		f (\xi, \beta M (\xi; u))
		\,
		d\xi
	\right)^{1 / (p - 1)}
\end{align*}
from which, by the inequalities
\begin{equation}
	M (r; u) - M (r_0; u)
	\ge
	\frac{1}{2}
	\beta^{- 1 / 2}
	M (r_0; u)
	\ge
	\frac{1}{2}
	\beta^{- j / 2}
	M (r_j; u)
	\label{PL3.5.7}
\end{equation}
and
\begin{align}
	&
	\int_{r_{j+1}}^{r_j}
	\xi^{1+a}
	e^{
		- k
		\int_\xi^{r_j}
		b (\zeta)
		\,
		d\zeta
	}
	f (\xi, \beta M (\xi; u))
	\,
	d\xi
	\nonumber
	\\
	&
	\quad
	\ge
	\int_{r_{j+1}}^{r_j}
	\xi^{1+a}
	e^{
		- k
		\int_\xi^{r_0}
		b (\zeta)
		\,
		d\zeta
	}
	f (\xi, \beta M (\xi; u))
	\,
	d\xi
	\nonumber
	\\
	&
	\quad
	\ge
	\frac{1}{2^j}
	\int_{R_0}^{r_1}
	\xi^{1+a}
	e^{
		- k
		\int_\xi^{r_0}
		b (\zeta)
		\,
		d\zeta
	}
	f (\xi, \beta M (\xi; u))
	\,
	d\xi,
	\label{PL3.5.8}
\end{align}
we obtain
\begin{align*}
	&
	M (r; u) - M (r_0; u)
	\\
	&
	\quad
	\ge
	\gamma_{12}
	2^{- j / (p  - 1) - 1}
	\beta^{- j / 2}
	(
		r_0^{
			- (a - p + 2) / (p - 1)
		}
		-
		r^{
			- (a - p + 2) / (p - 1)
		}
	)
	\\
	&
	\quad
	\phantom{{}\ge}
	\times
	\left(
		\int_{R_0}^{r_1}
		\xi^{1+a}
		e^{
			- k
			\int_\xi^{r_0}
			b (\zeta)
			\,
			d\zeta
		}
		f (\xi, \beta M (\xi; u))
		\,
		d\xi
	\right)^{1 / (p - 1)}.
\end{align*}
The last formula immediately implies~\eqref{PL3.5.3}.

Let proposition~(2) be valid.
If~\eqref{PL3.5.5} holds, then 
$$
	M (r_j; u) - M (r_{j+1} + 0; u)
	\ge
	\gamma_2
	(r_j - r_{j+1})
	\left(
		\int_{
			r_{j+1}
		}^{
			r_j
		}
		e^{
			- k
			\int_\xi^{r_j}
			b (\zeta)
			\,
			d\zeta
		}
		f (\xi, \beta M (\xi; u))
		\,
		d\xi
	\right)^{1 / (p-1)}
$$
by Corollary~\ref{C3.1}. Therefore, taking into account~\eqref{PL3.5.2} 
and the fact that
$
	r_j - r_{j+1}
	>
	2^{-j}
	(r_0 - r_1)
	=
	2^{- j - 1}
	\sigma^{- 1 / 2}
	(r - r_0),
$
we have
\begin{align*}
	&
	M (r_j; u) - M (r_{j+1} + 0; u)
	\\
	&
	\quad
	\ge
	\frac{
		\gamma_2 
		(p - 1)
	}{
		2^{j + 1} \sigma^{1/2} (a - p +2)
	}
	(
		r_0^{
			- (a - p + 2) / (p -1)
		}
		-
		r^{
			- (a - p + 2) / (p -1)
		}
	)
	\\
	&
	\quad
	\phantom{{}\ge}
	\times
	\left(
		\int_{
			r_{j+1}
		}^{
			r_j
		}
		\xi^{1+a}
		e^{
			- k
			\int_\xi^{r_j}
			b (\zeta)
			\,
			d\zeta
		}
		f (\xi, \beta M (\xi; u))
		\,
		d\xi
	\right)^{1 / (p-1)}.
\end{align*}
Combining this with inequalities~\eqref{PL3.5.7} and~\eqref{PL3.5.8},
one can show that
\begin{align*}
	&
	M (r; u) - M (r_0; u)
	\\
	&
	\quad
	\ge
	\frac{
		\gamma_2
		(p - 1)
	}{
		2^{j p / (p - 1) + 2}
		\beta^{j / 2}
		\sigma^{1/2} 
		(a - p +2)
	}
	(
		r_0^{
			- (a - p + 2) / (p -1)
		}
		-
		r^{
			- (a - p + 2) / (p -1)
		}
	)
	\\
	&
	\quad
	\phantom{{}\ge}
	\times
	\left(
		\int_{
			R_0
		}^{
			r_1
		}
		\xi^{1+a}
		e^{
			- k
			\int_\xi^{r_0}
			b (\zeta)
			\,
			d\zeta
		}
		f (\xi, \beta M (\xi; u))
		\,
		d\xi
	\right)^{1 / (p-1)},
\end{align*}
whence~\eqref{PL3.5.3} follows.
On the other hand, if~\eqref{PL3.5.5} is not fulfilled, 
then $r_{j+1} > R_0$ and
\begin{align}
	&
	\int_{
		R_0
	}^{
		r_{j+1}
	}
	\xi^{1+a}
	e^{
		- k
		\int_\xi^{
			r_{j+1}
		}
		b (\zeta)
		\,
		d\zeta
	}
	f (\xi, \beta M (\xi; u))
	\,
	d\xi
	\nonumber
	\\
	&
	\quad
	\ge
	\int_{
		R_0
	}^{
		r_{j+1}
	}
	\xi^{1+a}
	e^{
		- k
		\int_\xi^{r_0}
		b (\zeta)
		\,
		d\zeta
	}
	f (\xi, \beta M (\xi; u))
	\,
	d\xi
	\nonumber
	\\
	&
	\quad
	\ge
	\frac{1}{2}
	\int_{R_0}^{r_j}
	\xi^{1+a}
	e^{
		- k
		\int_\xi^{r_0}
		b (\zeta)
		\,
		d\zeta
	}
	f (\xi, \beta M (\xi; u))
	\,
	d\xi
	\nonumber
	\\
	&
	\quad
	\ge
	\frac{1}{2^j}
	\int_{R_0}^{r_1}
	\xi^{1+a}
	e^{
		- k
		\int_\xi^{r_0}
		b (\zeta)
		\,
		d\zeta
	}
	f (\xi, \beta M (\xi; u))
	\,
	d\xi.
	\label{PL3.5.11}
\end{align}
By the induction hypothesis, we have
\begin{align*}
	&
	M (r_j; u) - M (r_{j+1}; u)
	\\
	&
	\quad
	\ge
	\frac{
		2^{p / (p - 1)} (p - 1)
	}{
		a - p +2
	}
	(
		r_{j+1}^{
			- (a - p + 2) / (p -1)
		}
		-
		r_j^{
			- (a - p + 2) / (p -1)
		}
	)
	\\
	&
	\quad
	\phantom{{}\ge}
	\times
	\left(
		\alpha
		\int_{
			R_0
		}^{
			r_{j+1}
		}
		\xi^{1+a}
		e^{
			- k
			\int_\xi^{
				r_{j+1}
			}
			b (\zeta)
			\,
			d\zeta
		}
		f (\xi, \beta M (\xi; u))
		\,
		d\xi
	\right)^{1 / (p-1)};
\end{align*}
therefore,
\begin{align*}
	&
	M (r_j; u) - M (r_{j+1}; u)
	\\
	&
	\quad
	\ge
	\frac{
		2^{p / (p - 1)}
		(p - 1)
	}{
		2^{j / (p - 1)}
		(a - p +2)
	}
	(
		r_{j+1}^{
			- (a - p + 2) / (p -1)
		}
		-
		r_j^{
			- (a - p + 2) / (p -1)
		}
	)
	\nonumber
	\\
	&
	\quad
	\phantom{{}\ge}
	\times
	\left(
		\alpha
		\int_{
			R_0
		}^{
			r_1
		}
		\xi^{1+a}
		e^{
			- k
			\int_\xi^{
				r_0
			}
			b (\zeta)
			\,
			d\zeta
		}
		f (\xi, \beta M (\xi; u))
		\,
		d\xi
	\right)^{1 / (p-1)}.
\end{align*}
Combining this with~\eqref{PL3.5.7} and the relation
\begin{align*}
	r_{j+1}^{
		- (a - p + 2) / (p -1)
	}
	-
	r_j^{
		- (a - p + 2) / (p -1)
	}
	&
	\ge
	\frac{a - p + 2}{p - 1}
	(r_j - r_{j+1})
	r_j^{
		- (a + 1) / (p - 1)
	}
	\\
	&
	\ge
	\frac{
		a - p + 2
	}{
		2^j
		(p - 1)
	}
	(r_0 - r_1)
	r_0^{
		- (a + 1) / (p - 1)
	}
	\\
	&
	=
	\frac{
		a - p + 2
	}{
		2^{j + 1} \sigma^{1/2}
		(p - 1)
	}
	(r - r_0)
	r_0^{
		- (a + 1) / (p - 1)
	}
	\\
	&
	\ge
	2^{ - j - 1} \sigma^{- 1 / 2}
	(
		r_0^{
			- (a - p + 2) / (p -1)
		}
		-
		r^{
			- (a - p + 2) / (p -1)
		}
	),
\end{align*}
one can establish the validity of the estimate
\begin{align*}
	&
	M (r; u) - M (r_0; u)
	\\
	&
	\quad
	\ge
	\frac{
		2^{p / (p - 1)}
		(p - 1)
	}{
		2^{j p / (p - 1) + 2}
		\beta^{j/2}
		\sigma^{1/2}
		(a - p +2)
	}
	(
		r_0^{
			- (a - p + 2) / (p -1)
		}
		-
		r^{
			- (a - p + 2) / (p -1)
		}
	)
	\\
	&
	\quad
	\phantom{{}\ge}
	\times
	\left(
		\alpha
		\int_{
			R_0
		}^{
			r_1
		}
		\xi^{1+a}
		e^{
			- k
			\int_\xi^{
				r_0
			}
			b (\zeta)
			\,
			d\zeta
		}
		f (\xi, \beta M (\xi; u))
		\,
		d\xi
	\right)^{1 / (p-1)}
\end{align*}
from which~\eqref{PL3.5.3} follows again.

Now, let proposition~(3) be valid.
It is obvious that
$$
	r_1 - r_{j+1}
	=
	\sum_{i = 1}^j
	(r_i - r_{i+1})
	\le
	\sum_{i = 1}^j
	2^{-i}
	(r_0 - r_1)
	\le
	r_0 - r_1.
$$
In particular, $\sigma^{1/2} r_{j+1} \ge r_0$.
Consequently, Corollary~\ref{C3.2} implies the inequality
\begin{align*}
	&
	M (r; u) - M (r_0 + 0; u)
	\\
	&
	\quad
	\ge
	\gamma_3
	(r - r_0)
	\left(
		\frac{
			r_0 - r_j
		}{
			r_j - r_{j+1}
		}
		\int_{
			r_{j+1}
		}^{
			r_j
		}
		e^{
			- k
			\int_\xi^{r_0}
			b (\zeta)
			\,
			d\zeta
		}
		f (\xi, \beta M (\xi; u))
		\,
		d\xi
	\right)^{1 / (p - 1)},
\end{align*}
whence in accordance with~\eqref{PL3.5.2}, \eqref{PL3.5.8} and the fact that
$r_j - r_{j+1} \le 2^{-j} (r_0 - r_1) \le 2^{-j} (r_0 - r_j)$ 
we obtain
\begin{align*}
	&
	M (r; u) - M (r_0 + 0; u)
	\\
	&
	\quad
	\ge
	\frac{
		\gamma_3
		(p - 1)
	}{
		a - p +2
	}
	(
		r_0^{
			- (a - p + 2) / (p -1)
		}
		-
		r^{
			- (a - p + 2) / (p -1)
		}
	)
	\\
	&
	\quad
	\phantom{{}\ge}
	\times
	\left(
		\int_{
			R_0
		}^{
			r_1
		}
		\xi^{1+a}
		e^{
			- k
			\int_\xi^{r_0}
			b (\zeta)
			\,
			d\zeta
		}
		f (\xi, \beta M (\xi; u))
		\,
		d\xi
	\right)^{1 / (p-1)}.
\end{align*}
This obviously implies~\eqref{PL3.5.3}.

Finally, let proposition~(4) be valid.
If $r_{j+1} = R_0$, then the right-hand side of~\eqref{PL3.5.3} is equal to zero;
therefore, estimate~\eqref{PL3.5.3} is trivial.
Thus, one can assume that $r_{j+1} > R_0$. In this case, we have
$M (r_{j+1}; u) \le \beta^{1/2} M (r_j; u) \le M (r_{j+1} + 0; u)$
and Lemma~\ref{L3.4} allows us to assert that
\begin{align*}
	&
	M (r_j; u) - M (r_{j+1}; u)
	\\
	&
	\quad
	\ge
	2^{- p / (p - 1)}
	(1 - \sigma^{- 1/2})
	\beta^{- 1/2}
	r_{j+1}^{
		- (a - p + 2) / (p -1)
	}
	\\
	&
	\quad
	\phantom{{}\ge}
	\times
	\left(
		\alpha
		\int_{
			R_0
		}^{
			r_{j+1}
		}
		\xi^{1+a}
		e^{
			- k
			\int_\xi^{
				r_{j+1}
			}
			b (\zeta)
			\,
			d\zeta
		}
		f (\xi, \beta M (\xi; u))
		\,
		d\xi
	\right)^{1 / (p-1)}.
\end{align*}
Combining the last relation with~\eqref{PL3.5.7} and~\eqref{PL3.5.11}, 
we obtain
\begin{align*}
	&
	M (r; u) - M (r_0; u)
	\\
	&
	\quad
	\ge
	2^{- (p + j) / (p - 1) - 1}
	(1 - \sigma^{- 1/2})
	\beta^{- (j + 1) / 2}
	r_{j+1}^{
		- (a - p + 2) / (p -1)
	}
	\\
	&
	\quad
	\phantom{{}\ge}
	\times
	\left(
		\alpha
		\int_{
			R_0
		}^{
			r_1
		}
		\xi^{1+a}
		e^{
			- k
			\int_\xi^{
				r_0
			}
			b (\zeta)
			\,
			d\zeta
		}
		f (\xi, \beta M (\xi; u))
		\,
		d\xi
	\right)^{1 / (p-1)},
\end{align*}
whence~\eqref{PL3.5.3} follows at once. The proof is completed.
\end{proof}

\begin{proof}[Proof of Theorem~$\ref{T2.2}$]
The proof is by induction over the minimal positive integer $N$
such that $M (R_0 + 0; u) \ge \beta^{N / 2} M (r; u)$.
If $N = 1$, then~\eqref{T2.2.1} follows from Lemma~\ref{L3.2}.
Assume that Theorem~\ref{T2.2} is already proved
for all $N \le N_0$, where $N_0$ is some positive integer.
Let us prove it for $N = N_0 + 1$.
Put
$$
	r_0
	=
	\inf
	\{
		\xi \in (R_0, r)
		:
		M (\xi; u) > \beta^{1 / 2} M (r; u)
	\}.
$$
We have $R_0 < r_0 < r$ and, moreover,
${M (r_0; u)} \le \beta^{1/2} {M (r; u)} \le {M (r_0 + 0; u)}$.

By the induction hypothesis,
\begin{align}
	&
	M (r_0; u) - M (R_0+0; u)
	\nonumber
	\\
	&
	\quad
	\ge
	\int_{R_0}^{r_0}
	dt
	\,
	\left(
		\frac{\alpha}{t^{1+a}}
		\int_{R_0}^t
		\xi^{1+a}
		e^{
			- k
			\int_\xi^t
			b (\zeta)
			\,
			d\zeta
		}
		f (\xi, \beta M (\xi;  u))
		\,
		d\xi
	\right)^{1 / (p-1)}.
	\label{PT2.2.1}
\end{align}
At the same time, it can be shown that
\begin{align}
	&
	M (r; u) - M (r_0; u)
	\nonumber
	\\
	&
	\quad
	\ge
	\int_{r_0}^r
	dt
	\,
	\left(
		\frac{\alpha}{t^{1+a}}
		\int_{R_0}^t
		\xi^{1+a}
		e^{
			- k
			\int_\xi^t
			b (\zeta)
			\,
			d\zeta
		}
		f (\xi, \beta M (\xi;  u))
		\,
		d\xi
	\right)^{1 / (p-1)}.
	\label{PT2.2.2}
\end{align}
Really, the right-hand side of the last expression satisfies the inequality
\begin{align*}
	&
	\int_{r_0}^r
	dt
	\,
	\left(
		\frac{\alpha}{t^{1+a}}
		\int_{R_0}^t
		\xi^{1+a}
		e^{
			- k
			\int_\xi^t
			b (\zeta)
			\,
			d\zeta
		}
		f (\xi, \beta M (\xi; u))
		\,
		d\xi
	\right)^{1 / (p-1)}
	\\
	&
	\quad
	\le
	2^{1 / (p - 1)}
	\int_{r_0}^r
	dt
	\,
	\left(
		\frac{\alpha}{t^{1+a}}
		\int_{r_0}^t
		\xi^{1+a}
		e^{
			- k
			\int_\xi^t
			b (\zeta)
			\,
			d\zeta
		}
		f (\xi, \beta M (\xi; u))
		\,
		d\xi
	\right)^{1 / (p-1)}
	\\
	&
	\quad
	\phantom{{}\le}
	+
	2^{1 / (p - 1)}
	\int_{r_0}^r
	dt
	\,
	\left(
		\frac{\alpha}{t^{1+a}}
		\int_{R_0}^{r_0}
		\xi^{1+a}
		e^{
			- k
			\int_\xi^t
			b (\zeta)
			\,
			d\zeta
		}
		f (\xi, \beta M (\xi; u))
		\,
		d\xi
	\right)^{1 / (p-1)}.
\end{align*}
Thus, formula~\eqref{PT2.2.2} will be proved if we succeed in proving
the estimates
\begin{align}
	&
	M (r; u) - M (r_0 + 0; u)
	\nonumber
	\\
	&
	\quad
	\ge
	2^{p / (p - 1)}
	\int_{r_0}^r
	dt
	\,
	\left(
		\frac{\alpha}{t^{1+a}}
		\int_{r_0}^t
		\xi^{1+a}
		e^{
			- k
			\int_\xi^t
			b (\zeta)
			\,
			d\zeta
		}
		f (\xi, \beta M (\xi; u))
		\,
		d\xi
	\right)^{1 / (p-1)}
	\label{PT2.2.3}
\end{align}
and
\begin{align}
	&
	M (r; u) - M (r_0; u)
	\nonumber
	\\
	&
	\quad
	\ge
	\frac{
		2^{p / (p - 1)}
		(p - 1)
	}{
		a - p +2
	}
	(
		r_0^{
			- (a - p + 2) / (p - 1)
		}
		-
		r^{
			- (a - p + 2) / (p - 1)
		}
	)
	\nonumber
	\\
	&
	\quad
	\phantom{{}=}
	\times
	\left(
		\alpha
		\int_{R_0}^{r_0}
		\xi^{1+a}
		e^{
			- k
			\int_\xi^{r_0}
			b (\zeta)
			\,
			d\zeta
		}
		f (\xi, \beta M (\xi; u))
		\,
		d\xi
	\right)^{1 / (p-1)}
	\nonumber
	\\
	&
	\quad
	\ge
	2^{p / (p - 1)}
	\int_{r_0}^r
	dt
	\,
	\left(
		\frac{\alpha}{t^{1+a}}
		\int_{R_0}^{r_0}
		\xi^{1+a}
		e^{
			- k
			\int_\xi^t
			b (\zeta)
			\,
			d\zeta
		}
		f (\xi, \beta M (\xi; u))
		\,
		d\xi
	\right)^{1 / (p-1)}.
	\label{PT2.2.4}
\end{align}
Estimate~\eqref{PT2.2.3} is a consequence of Lemma~\ref{L3.2}, 
whereas~\eqref{PT2.2.4} can be obtained by Lemma~\ref{L3.4} if
$\sigma^{1/2} r_0 \le r$
or by Lemma~\ref{L3.5} if
$\sigma^{1/2} r_0 > r$.

To compete the proof it remains to sum inequalities~\eqref{PT2.2.1}
and~\eqref{PT2.2.2}.
\end{proof}

\section{Proof of Lemma 3.1}\label{S4}

As in the previous section, we assume that $u$ is a non-negative solution of 
problem~\eqref{1.1}, \eqref{1.3} and, moreover, ${M (\cdot; u )}$ is a 
non-decreasing function on the interval $(R_0, R_1)$ satisfying 
condition~\eqref{T2.1.1}.

\begin{Lemma}\label{L4.1}
There is a symmetric $n \times n$-matrix $\| a_{ij} \|$
with measurable coefficients such that the function
$A = (A_1, \ldots, A_n)$
on the left in~$(\ref{1.1})$ can be written as follows:
$$
	A_i (x, \xi)
	=
	\sum_{j = 1}^n
	a_{ij} (x, \xi)
	|\xi|^{p-2}
	\xi_j,
	\quad
	i = 1, \ldots, n,
$$
for almost all $x \in \Omega_{R_0, R_1}$ and for all $\xi \in {\mathbb R}^n$.
In so doing, we have
\begin{equation}
	\lambda_1
	|\zeta|^2
	\le
	\sum_{i,j = 1}^n
	a_{ij} (x, \xi) \zeta_i \zeta_j
	\le
	\lambda_2
	|\zeta|^2
	\label{L4.1.1}
\end{equation}
for almost all $x \in \Omega_{R_0, R_1}$ and for all $\xi, \zeta \in {\mathbb R}^n$,
where the constants $\lambda_1 > 0$ and $\lambda_2 > 0$ depend only on
$C_1$ and $C_2$.
\end{Lemma}

The proof is given in~\cite[Lemma~4.1]{Me}.

\bigskip

We put
$$
	q (x, \xi)
	=
	\left(
	\sum_{i,j=1}^n
	q_{ij} (x) \xi_i \xi_j
	\right)^{
		(p - 2) / 2
	},
$$
where
$$
	q_{ij} (x)
	=
	\left\{
		\begin{array}{ll}
			a_{ij} (x, D u),
			&
			x
			\in
			\Omega_{R_0, R_1},
			\\
			(\lambda_1 + \lambda_2) \delta_{ij} / 2,
			&
			x
			\in
			{\mathbb R}^n
			\setminus
			\Omega_{R_0, R_1}.
		\end{array}
	\right.
$$
Also let
\begin{equation}
	Q_i (x, \xi)
	=
	h (x)
	q (x, \xi)
	\sum_{j = 1}^n
	q_{ij} (x)
	\xi_j,
	\quad
	i = 1, \ldots, n,
	\label{4.3}
\end{equation}
where
$
	h (x)
	=
	|D u|^{p-2}
	/
	q (x, D u)
$
for all
$x \in \Omega_{R_0, R_1}$
such that
$D u (x) \ne 0$
and
$h (x) =	(\lambda_1^{(2-p)/2} + \lambda_2^{(2-p)/2}) / 2$
for all other $x \in {\mathbb R}^n$.
In the case of $\xi = 0$, we assume that the right-hand side of~\eqref{4.3}
is equal to zero.

Relation~(\ref{L4.1.1}) implies the inequalities
\begin{equation}
	\min
	\{
		\lambda_1^{(2-p)/2},
		\lambda_2^{(2-p)/2}
	\}
	\le
	h (x)
	\le
	\max
	\{
		\lambda_1^{(2-p)/2},
		\lambda_2^{(2-p)/2}
	\}
	\label{4.1h}
\end{equation}
and
\begin{equation}
	\lambda_1
	|\xi|^2
	\le
	\sum_{i,j = 1}^n
	q_{ij} (x) \xi_i \xi_j
	\le
	\lambda_2
	|\xi|^2
	\label{4.1q}
\end{equation}
for almost all $x \in {\mathbb R}^n$ and for all $\xi \in {\mathbb R}^n$.

From Lemma~\ref{L4.1}, it follows that
\begin{equation}
	\diver Q (x, D u)
	\ge
	F (x, u, D u)
	\quad
	\mbox{in }
	\Omega_{R_0, R_1},
	\label{4.2}
\end{equation}
where $Q = (Q_1, \ldots, Q_n)$. In addition, we obtain
$$
	(
		Q (x, \xi)
		-
		Q (x, \zeta)
	)
	(\xi - \zeta)
	>
	0
$$
for almost all $x \in {\mathbb R}^n$ and for all 
$\xi, \zeta \in {\mathbb R}^n$,
$\xi \ne \zeta$.

Let $\omega_1$ and $\omega_2$ be open subsets of ${\mathbb R}^n$ and
$v \in W_{p, loc}^1 (\omega_1 \cap \omega_2)$.
We say that
$$
    \left.
    v
    \right|_{
        \omega_2 \cap \partial \omega_1
    }
    =
    0
$$
if
$
    \varphi
     v
    \in
    {
    	\stackrel{\rm \scriptscriptstyle o}{W}\!\!{}_p^1 
    	(
    		\omega_1 \cap \omega_2
    	)
    }
$
for any $\varphi \in C_0^\infty (\omega_2)$.
Analogously,
$$
    \left.
    v
    \right|_{
        \omega_2 \cap \partial \omega_1
    }
    \le
    0
$$
if
$
    \varphi
    \max \{ v, 0 \}
    \in
    {
    	\stackrel{\rm \scriptscriptstyle o}{W}\!\!{}_p^1 
    	(
    		\omega_1 \cap \omega_2
    	)
    }
$
for any $\varphi \in C_0^\infty (\omega_2)$.

\begin{Lemma}\label{L4.2}
Suppose that
$v \in W_p^1 (\omega_1 \cap \omega_2)$
is a solution of the problem
$$
	\diver Q (x, D v)
	\ge
	g (x)
	\quad
	\mbox{in } \omega_1 \cap \omega_2,
	\quad
	\left.
	v
	\right|_{
		\omega_2 \cap \partial \omega_1
	}
	\le
	0,
$$
where $\omega_1$ and $\omega_2$ are bounded open subsets of ${\mathbb R}^n$
and $g \in L_{p / (p - 1)} (\omega_1 \cap \omega_2)$ is some function.
We denote:
$\omega_0 = \{ x \in \omega_1 \cap \omega_2 : v (x) > 0 \}$,
$$
	v_0 (x)
	=
	\left\{
		\begin{array}{ll}
			v (x),
			&
			x \in \omega_0,
			\\
			0,
			&
			x \in \omega_2 \setminus \omega_0
		\end{array}
	\right.
$$
and
$$
	g_0 (x)
	=
	\left\{
		\begin{array}{ll}
			g (x),
			&
			x \in \omega_0,
			\\
			0,
			&
			x \in \omega_2 \setminus \omega_0.
		\end{array}
	\right.
$$
Then
$$
	\diver Q (x, D v_0)
	\ge
	g_0 (x)
	\quad
	\mbox{in } \omega_2.
$$
\end{Lemma}

\begin{Lemma}\label{L4.3}
For every non-negative function
$
	w
	\in
	W_p^1 (B_r^y)
	\cap
	L_\infty (B_r^y),
$
$r > 0$, $y \in {\mathbb R}^n$,
there exists a function
$
	\psi
	\in
	{
		\stackrel{\rm \scriptscriptstyle o}{W}\!\!{}_p^1 
		(B_r^y)
	}
	\cap
	L_\infty (B_r^y)
$
such that
$0 \le \psi \le 1$ almost everywhere on $B_r^y$,
$\psi = 1$ almost everywhere on $B_{r / 2}^y$ and, moreover,
$$
	\esssup_{B_r^y} w^{p-1}
	\ge
	- C
	r^{p - n}
	\int_{B_r^y}
	Q (x, D w)
	D \psi
	\,
	dx,
$$
where the constant $C > 0$ depends only on $n$, $p$, $C_1$, and $C_2$.
\end{Lemma}

The proof of Lemmas~\ref{L4.2} and~\ref{L4.3} is given in~\cite[Lemmas~4.2 and~4.4]{Me}.

\begin{Proposition}[maximum principle]\label{P4.1}
Suppose that
$v \in W_p^1 (\omega) \cap L_\infty (\omega)$,
where $\omega$ is an open bounded subset of ${\mathbb R}^n$ 
with an infinitely smooth boundary and, moreover,
\begin{equation}
	\diver Q (x, D v) 
	+ 
	H (x) 
	|D v|^{p - 1}
	\ge
	0
	\label{P4.1.1}
\end{equation}
in $\omega$ for some funtion $H \in L_\infty (\omega)$. Then
\begin{equation}
	\esssup
	\left.
		v
	\right|_{
		\partial \omega
	}
	=
	\esssup_\omega
	v,
	\label{P4.1.2}
\end{equation}
where the restriction of $v$ to $\partial \omega$
is understood in the sense of the trace and
the $\esssup$ in the left-hand side of~\eqref{P4.1.2} is with respect to
$(n-1)$-dimensional Lebesgue measure on $\partial \omega$.
\end{Proposition}

\begin{Proposition}\label{P4.2}
Let $v \in W_p^1 ((0, l)^n)$, $l > 0$. If
$
	\mes
	\{
		x \in (0, l)^n
		:
		v (x) = 0
	\}
	\ge
	l^n / 2,
$
then
$$
	\int_{(0, l)^n}
	|v|^p
	\,
	dx
	\le
	C
	l^p
	\int_{(0, l)^n}
	|D v|^p
	\,
	dx,
$$
where the constant $C > 0$ depends only on $n$ and $p$.
\end{Proposition}

\begin{Proposition}[Moser's inequality]\label{P4.3}
Assume that $v \in W_p^1 (B_r^y) \cap L_\infty (B_r^y)$ 
is a non-negative solution of inequality~\eqref{P4.1.1}
in the ball $B_r^y$, $r > 0$, $y \in {\mathbb R}^n$,
where $H \in L_\infty (B_r^y)$ satisfies the condition
\begin{equation}
	r
	\esssup_{
		B_r^y
	}
	|H|
	\le 
	1.
	\label{P4.3.1}
\end{equation}
Then
$$
	\esssup_{
		B_{r/2}^y
	}
	v^p
	\le
	C
	r^{-n}
	\int_{
		B_r^y
	}
	v^p
	\,
	dx,
$$
where the constant $C > 0$ depends only on $n$, $p$, $C_1$, and $C_2$.
\end{Proposition}

We omit the proof of Propositions~\ref{P4.1}--\ref{P4.3}
as it is pretty standard (see~\cite{LU}, \cite{Serrin}).

\begin{Corollary}\label{C4.1}
For all $r \in (R_0, R_1)$
\begin{equation}
	M (r; u)
	=
	\esssup_{
		\Omega_{R_0, r}
	}
	u.
	\label{C4.1.1}
\end{equation}
\end{Corollary}

\begin{proof}
Without loss of generality it can be assumed that $R_0 > 0$; otherwise 
we pass in~\eqref{C4.1.1} to the limit as $R_0 \to +0$.
Take some $r \in (R_0, R_1)$.
By~\eqref{1.5} and~\eqref{4.2}, the function $u$ satisfies the inequality
$$
	\diver Q (x, D u) 
	+ 
	b (|x|) 
	|D u|^{p - 1}
	\ge
	0
	\quad
	\mbox{in }
	\Omega_{R_0, r}.
$$
Let us put 
$$
	u_0 (x)
	=
	\left\{
		\begin{array}{ll}
			u (x),
			&
			x \in \omega_0,
			\\
			0,
			&
			x \in B_{R_0, r} \setminus \omega_0,
		\end{array}
	\right.
$$
where 
$
	\omega_0 
	= 
	\{ 
		x \in \Omega_{R_0, r}
		:
		u (x) > 0
	\}
$
and
$
	B_{R_0, r} 
	= 
	\{
		x \in {\mathbb R}^n
		:
		R_0 < |x| < r
	\}.
$
Lemma~\ref{L4.2} obviously implies that
$$
	\diver Q (x, D u_0) 
	+ 
	b (|x|) 
	|D u_0|^{p - 1}
	\ge
	0
	\quad
	\mbox{in }
	B_{R_0, r}.
$$
Thus,
$$
	\esssup
	\left.
		u_0
	\right|_{
		\partial B_{R_0, r}
	}
	=
	\esssup_{
		B_{R_0, r}
	}
	u_0
$$
according to Proposition~\ref{P4.1}.
To complete the proof it remains to notice that
$$
	M (r; u)
	=
	\esssup
	\left.
		u_0
	\right|_{
		\partial B_{R_0, r}
	}
$$
and
$$
ô	\esssup_{
		B_{R_0, r}
	}
	u
	=
	\esssup_{
		B_{R_0, r}
	}
	u_0.
$$
\end{proof}

\begin{Lemma}\label{L4.4}
Let the hypotheses of Proposition~$\ref{P4.3}$ be fulfilled, then
for any $\varepsilon > 0$ there exists a real number $\delta > 0$ 
depending only on $n$, $p$, $C_1$, $C_2$, and $\varepsilon$
such that the relation
$
	\mes
	\{
		x \in B_r^y
		:
		v (x) > 0
	\}
	\le
	\delta
	r^n	
$
implies the estimate
$$
	\esssup_{
		B_{r/8}^y
	}
	v
	\le
	\varepsilon
	\esssup_{
		B_r^y
	}
	v.
$$
\end{Lemma}

\begin{proof}
In the case of
$
	\esssup\nolimits_{
		B_r^y
	}
	v
	=
	0,
$
Lemma~\ref{L4.4} is trivial.
Without loss of generality it can be assumed that
$
	\esssup\nolimits_{
		B_r^y
	}
	v
	=
	1;
$
otherwise we replace the function $v$ by 
$
	v
	/
	\esssup\nolimits_{
		B_r^y
	}
	v.
$
Also it can be assumed that $r = 1$ and $y = 0$; 
otherwise we use the change of variables.

Take a non-negative function $\eta \in C_0^\infty (B_1)$ such that
$
	\left.
		\eta
	\right|_{
		B_{1/2}
	}
	=
	1.
$
From~\eqref{P4.1.1}, we have
$$
	-
	\int_{
		B_1
	}
	Q (x, D v) 
	D (\eta^p v)
	\,
	dx
	+ 
	\int_{
		B_1
	}
	H (x) 
	|D v|^{p - 1}
	\eta^p v
	\,
	dx
	\ge
	0
$$
or, in other words,
$$
	\int_{
		B_1
	}
	\eta^p
	Q (x, D v) 
	D v
	\,
	dx
	\le
	\int_{
		B_1
	}
	H (x) 
	|D v|^{p - 1}
	\eta^p v
	\,
	dx
	-
	p
	\int_{
		B_1
	}
	\eta^{p-1}
	v
	Q (x, D v) 
	D \eta
	\,
	dx.
$$
Using Young's inequality, one can show that
$$
	\int_{
		B_1
	}
	|H (x)|
	|D v|^{p - 1}
	\eta^p v
	\,
	dx
	\le
	\mu
	\int_{
		B_1
	}
	\eta^p
	|D v|^p
	\,
	dx
	+
	\mu_*
	\int_{
		B_1
	}
	\eta^p
	v^p
	\,
	dx
$$
and
$$
	\int_{
		B_1
	}
	\eta^{p-1}
	v
	Q (x, D v) 
	D \eta
	\,
	dx
	\le
	\mu
	\int_{
		B_1
	}
	\eta^p
	|Q (x, D v)|^{p / (p - 1)}
	\,
	dx
	+
	\mu_*
	\int_{
		B_1
	}
	|D \eta|^p
	v^p
	\,
	dx
$$
for all real numbers $\mu > 0$, where the constant $\mu_* > 0$ 
depends only on $p$ and $\mu$.
On the other hand, in accordance with~\eqref{4.1h} and~\eqref{4.1q}
there are constants $\varkappa > 0$ and $\varkappa_* > 0$ 
depending only on $n$, $p$, $C_1$, and $C_2$ such that
$$
	\varkappa
	\int_{
		B_1
	}
	\eta^p
	|D v|^p
	\,
	dx
	\le
	\int_{
		B_1
	}
	\eta^p
	Q (x, D v) 
	D v
	\,
	dx
$$
and
$$
	\int_{
		B_1
	}
	\eta^p
	|Q (x, D v)|^{p / (p - 1)}
	\,
	dx
	\le
	\varkappa_*
	\int_{
		B_1
	}
	\eta^p
	|D v|^p
	\,
	dx.
$$
Hence, choosing sufficiently small $\mu > 0$, we obtain the estimate
$$
	\int_{
		B_{1/2}
	}
	|D v|^p
	\,
	dx
	\le
	\int_{
		B_1
	}
	\eta^p
	|D v|^p
	\,
	dx
	\le
	\tau
	\int_{
		B_1
	}
	v^p
	\,
	dx,
$$
where the constant $\tau > 0$ depends only on $n$, $p$, $C_1$, and $C_2$.

Let
$
	\mes
	\{
		x \in B_1
		:
		v (x) > 0
	\}
	\le
	\delta
$
for some $\delta > 0$ such that $\delta < 8^{- n} n^{- n / 2}$.
There exists a finit family of the disjoint open cubs 
$J_i$, $i = 1, 2, \ldots, N$, 
with the edge length equal to $2 \delta^{1/n}$ such that
$
	B_{1/4}
	\subset
	\bigcup_{i=1}^N
	\overline{J}_i
	\subset
	B_{1/2}.
$
According to Proposition~\ref{P4.2}, we have
$$
	\int_{J_i}
	v^p
	\,
	dx
	\le
	\zeta
	\delta^{p / n}
	\int_{J_i}
	|D v|^p
	\,
	dx,
	\quad
	i = 1,2,\ldots,N,
$$
where the constant $\zeta > 0$ depends only on $n$ and $p$; 
therefore,
$$
	\int_{
		B_{1/4}
	}
	v^p
	\,
	dx
	\le
	\zeta
	\delta^{p / n}
	\int_{
		B_{1/2}
	}
	|D v|^p
	\,
	dx.
$$
At the same time, from Moser's inequality, it follows that
$$
	\esssup_{
		B_{1/8}
	}
	v^p
	\le
	\theta
	\int_{
		B_{1/4}
	}
	v^p
	\,
	dx,
$$
where the constant $\theta > 0$ depends only on $n$, $p$, $C_1$, and $C_2$.
Thus, we obtain
$$
	\esssup_{
		B_{1/8}
	}
	v^p
	\le
	\tau
	\zeta
	\theta
	\delta^{p / n}
	\int_{
		B_1
	}
	v^p
	\,
	dx
	\le
	\tau
	\zeta
	\theta
	\delta^{p / n}
	\mes B_1.
$$
To complete the proof it remains to take the real number $\delta > 0$
satisfying the condition
$
	\tau
	\zeta
	\theta
	\delta^{p / n}
	\mes B_1
	\le
	\varepsilon^p.
$
\end{proof}

\begin{Lemma}\label{L4.5}
Let the hypotheses of Proposition~$\ref{P4.3}$ be fulfilled, then
\begin{equation}
	\int_{
		B_{r/2}^y
	}
	|H (x)|
	|D v|^{p-1}
	\,
	dx
	\le
	C
	r^{n - p}
	\esssup_{
		B_r^y
	}
	v^{p - 1},
	\label{L4.5.1}
\end{equation}
where the constant $C > 0$ depends only on $n$, $p$, $C_1$, and $C_2$.
\end{Lemma}

\begin{proof}
Applying Lemma~\ref{L4.3} with $w = v^{p / (p - 1)}$, we have
\begin{equation}
	\esssup_{B_r^y} v^p
	\ge
	- \mu
	r^{p - n}
	\int_{B_r^y}
	Q (x, D v)
	D \psi
	v
	\,
	dx
	\label{PL4.5.1}
\end{equation}
for some non-negative function
$
	\psi
	\in
	{
		\stackrel{\rm \scriptscriptstyle o}{W}\!\!{}_p^1 
		(B_r^y)
	}
	\cap
	L_\infty (B_r^y)
$
such that $0 \le \psi \le 1$ almost everywhere on $B_r^y$ and
$\psi = 1$ almost everywhere on $B_{r / 2}^y$,
where the constant $\mu > 0$ depends only on $n$, $p$, $C_1$, and $C_2$.
Put
$$
	I
	=
	\int_{B_r^y}
	|H (x)|
	|D v|^{p-1}
	\psi
	\,
	dx.
$$
By the H\"older inequality,
\begin{eqnarray}
	I
	&
	\le
	&
	\esssup_{B_r^y}
	|H|
	\int_{B_r^y}
	|D v|^{p-1}
	\psi
	\,
	dx
	\nonumber
	\\
	&
	\le
	&
	\esssup_{B_r^y}
	|H|
	\left(
		\int_{B_r^y}
		\psi
		\,
		dx
	\right)^{1/p}
	\left(
		\int_{B_r^y}
		|D v|^p
		\psi
		\,
		dx
	\right)^{( p - 1) / p};
	\nonumber
\end{eqnarray}
therefore, taking into account condition~\eqref{P4.3.1}, we obtain
\begin{equation}
	I
	\le
	r^{n / p - 1}
	(\mes B_1)^{1 / p}
	\left(
		\int_{B_r^y}
		|D v|^p
		\psi
		\,
		dx
	\right)^{( p - 1) / p}.
	\label{PL4.5.2}
\end{equation}
At the same time, relations~\eqref{4.1h} and~\eqref{4.1q} imply the estimate
\begin{equation}
	\varkappa
	\int_{B_r^y}
	|D v|^p
	\psi
	\,
	dx
	\le
	\int_{
		B_r^y
	}
	Q (x, D v) 
	D v
	\psi
	\,
	dx,
	\label{PL4.5.3}
\end{equation}
where the constant $\varkappa > 0$ depends only on $p$, $C_1$, and $C_2$.
Since the function $v$ is a non-negative solution of inequality~\eqref{P4.1.1} in $B_r^y$, 
we have
$$
	- \int_{
		B_r^y
	}
	Q (x, D v) 
	D (\psi v)
	\,
	dx
	+
	\int_{B_r^y}
	H (x)
	|D v|^{p-1}
	\psi
	v
	\,
	dx
	\ge
	0.
$$
Consequently,
\begin{eqnarray}
	\int_{
		B_r^y
	}
	Q (x, D v) 
	D v
	\psi
	\,
	dx
	&
	=
	&
	\int_{
		B_r^y
	}
	Q (x, D v) 
	D (\psi v)
	\,
	dx
	-
	\int_{
		B_r^y
	}
	Q (x, D v) 
	D \psi
	v
	\,
	dx
	\nonumber
	\\
	&
	\le
	&
	\int_{B_r^y}
	H (x)
	|D v|^{p-1}
	\psi
	v
	\,
	dx
	-
	\int_{
		B_r^y
	}
	Q (x, D v) 
	D \psi
	v
	\,
	dx.
	\nonumber
\end{eqnarray}
The last relation and~\eqref{PL4.5.1} allow us to assert that
$$
	\int_{
		B_r^y
	}
	Q (x, D v) 
	D v
	\psi
	\,
	dx
	\le
	I
	\esssup_{
		B_r^y
	}
	v
	+
	\frac{
		r^{n - p}
	}{
		\mu
	}
	\esssup_{
		B_r^y
	}
	v^p.
$$
Combining this with~\eqref{PL4.5.2} and~\eqref{PL4.5.3}, we obtain
\begin{equation}
	I
	\le
	\zeta
	r^{n/p - 1}
	\left(
		I
		\esssup_{
			B_r^y
		}
		v
		+
		r^{n - p}
		\esssup_{
			B_r^y
		}
		v^p
	\right)^{( p - 1) / p},
	\label{PL4.5.4}
\end{equation}
where the constant $\zeta > 0$ depends only on $n$, $p$, $C_1$, and $C_2$.

In the case of
\begin{equation}
	I
	\esssup_{
		B_r^y
	}
	v
	\ge
	r^{n - p}
	\esssup_{
		B_r^y
	}
	v^p,
	\label{PL4.5.5}
\end{equation}
formula~\eqref{PL4.5.4} enables one to establish the validity of the inequality
$$
	I
	\le
	2^{( p - 1) / p}
	\zeta
	r^{n/p - 1}
	\left(
		I
		\esssup_{
			B_r^y
		}
		v
	\right)^{( p - 1) / p}
$$
or, in other words,
$$
	I
	\le
	2^{p - 1}
	\zeta^p
	r^{n - p}
	\esssup_{
		B_r^y
	}
	v^{p - 1},
$$
whence~\eqref{L4.5.1} immediately follows.
On the other hand, if~\eqref{PL4.5.5} does not hold, 
then in accordance with~\eqref{PL4.5.4} we have
$$
	I
	\le
	2^{(p - 1) / p}
	\zeta
	r^{n - p}
	\esssup_{
		B_r^y
	}
	v^{p - 1}.
$$
This also implies~\eqref{L4.5.1}.
The lemma is completely proved.
\end{proof}

\begin{proof}[Proof of Lemma~$\ref{L3.1}$]
Let
$s \in [r_1 / \sigma, \sigma r_0] \cap (R_0, R_1)$,
$
	r 
	=
	\min
	\{
		1 / \lambda,
		(r_1 - r_0) / 4
	\}
$
and, moreover, $N$ be the maximal integer such that $N r < (r_1 - r_0) / 2$.
Put $\rho_i = r_0 + i r$, $i = 1, \ldots, N$. 
For each $i \in \{ 1, \ldots, N \}$ we take a point 
$y_i \in S_{\rho_i} \cap \Omega$
satisfying the condition
\begin{equation}
	\lim_{
		\varepsilon \to +0
	}
	\esssup_{
		B_\varepsilon^{y_i} \cap \Omega
	}
	u
	\ge
	M (\rho_i; u).
	\label{PL3.1.5}
\end{equation}

Assume that
\begin{equation}
	\esssup_{
		B_r^{y_i}
		\cap
		\Omega
	}
	u
	\ge
	(2 - \beta^{1/2})
	M (\rho_i; u)
	\label{PL3.1.1}
\end{equation}
for some $i \in \{ 1, \ldots, N \}$.
We denote
$
	\omega_i
	=
	\{
		x \in B_r^{y_i} \cap \Omega
		:
		u (x) > \beta^{1/2} M (\rho_i; u)
	\}
$
and
$$
	v_i (x)
	= 
	\left\{
		\begin{array}{ll}
			u (x) - \beta^{1/2} M (\rho_i; u),
			&
			x \in \omega_i,
			\\
			0,
			&
			x \in  B_r^{y_i} \setminus \omega_i.
		\end{array}
	\right.
$$
From Lemma~\ref{L4.2} and relation~\eqref{4.2}, it follows that
\begin{equation}
	\diver Q (x, D v_i)
	\ge
	F (x, u, D v_i)
	\chi_{\omega_i} (x)
	\quad
	\mbox{in }
	B_r^{y_i},
	\label{PL3.1.8}
\end{equation}
where $\chi_{\omega_i}$ is the characteristic function of $\omega_i$.
Therefore, in accordance with~\eqref{1.5} we obtain
\begin{equation}
	\diver Q (x, D v_i)
	+
	\lambda
	|D v_i|^{p - 1}
	\ge
	f (s, \beta^{1/2} M (\rho_i; u))
	\chi_{\omega_i} (x)
	\quad
	\mbox{in }
	B_r^{y_i}.
	\label{PL3.1.2}
\end{equation}

Since $\lambda r \le 1$, Lemma~\ref{L4.5} implies the inequality
\begin{equation}
	\mu
	r^{n - p}
	\esssup_{
		B_r^{y_i}
	}
	v_i^{p - 1}
	\ge
	\lambda
	\int_{
		B_{r/2}^{y_i}
	}
	|D v_i|^{p-1}
	\,
	dx.
	\label{PL3.1.3}
\end{equation}
where the constant $\mu > 0$ depends only on $n$, $p$, $C_1$, and $C_2$.
At the same time, by Lemma~\ref{L4.3}, there exists a function 
$
	\psi_i
	\in
	{
		\stackrel{\rm \scriptscriptstyle o}{W}\!\!{}_p^1 
		(B_{r/2}^{y_i})
	}
	\cap
	L_\infty (B_{r/2}^{y_i})
$
such that
$0 \le \psi_i \le 1$ almost everywhere on $B_{r/2}^{y_i}$,
$\psi_i = 1$ almost everywhere on $B_{r / 4}^{y_i}$ and, moreover,
\begin{equation}
	\esssup_{
		B_{r/2}^{y_i}
	} 
	v_i^{p-1}
	\ge
	-
	\nu
	r^{p - n}
	\int_{
		B_{r/2}^{y_i}
	}
	Q (x, D v_i)
	D \psi_i
	\,
	dx,
	\label{PL3.1.4}
\end{equation}
where the constant $\nu > 0$ depends only on $n$, $p$, $C_1$, and $C_2$.
According to~\eqref{PL3.1.2}, we have
$$
	- \int_{
		B_{r/2}^{y_i}
	}
	Q (x, D v_i)
	D \psi_i
	\,
	dx
	+
	\lambda
	\int_{
		B_{r/2}^{y_i}
	}
	|D v_i|^{p - 1}
	\psi_i
	\,
	dx
	\ge
	f (s, \beta^{1/2} M (\rho_i; u))
	\int_{
		\omega_i 
		\cap 
		B_{r/2}^{y_i}
	}
	\psi_i
	\,
	dx.
$$
Combining the last relation with~\eqref{PL3.1.3} and \eqref{PL3.1.4}, 
one can show that
\begin{equation}
	\esssup_{
		B_r^{y_i}
	} 
	v_i^{p-1}
	\ge
	\varkappa
	r^{p - n}
	\mes (\omega_i \cap B_{r/4}^{y_i})
	f (s, \beta^{1/2} M (\rho_i; u)),
	\label{PL3.1.6}
\end{equation}
where the constant $\varkappa > 0$ depends only on $n$, $p$, $C_1$, and $C_2$.
By condition~\eqref{PL3.1.5},
$$
	\lim_{
		\varepsilon \to +0
	}
	\esssup_{
		B_\varepsilon^{y_i}
	}
	v_i
	\ge
	(1 - \beta^{1/2})
	M (\rho_i; u)
	>
	0.
$$
Since 
$M (\rho_i; u) \ge M (r_0; u) \ge \beta^{1/2} M (r_1; u)$
and
$$
	M (r_1; u)
	=
	\esssup_{
		\Omega_{R_0, r_1}
	}
	u
	\ge
	\esssup_{
		B_r^{y_i}
	}
	v_i,
$$
this implies the estimate
$$
	\lim_{
		\varepsilon \to +0
	}
	\esssup_{
		B_\varepsilon^{y_i}
	}
	v_i
	\ge
	(1 - \beta^{1/2})
	\beta^{1/2}
	\esssup_{
		B_r^{y_i}
	}
	v_i.
$$
Therefore, by Lemma~\ref{L4.4}, there exists a real number $\delta > 0$
depending only on $n$, $p$, $C_1$, $C_2$, and $\beta$ such that
$\mes (\omega_i \cap B_{r/4}^{y_i}) \ge \delta r^n$.
At the same time, taking into account~\eqref{PL3.1.1}, we obtain
$$
	\esssup_{
		B_r^{y_i}
		\cap
		\Omega
	}
	u
	-
	M (\rho_i; u)
	\ge
	(1 - \beta^{1/2})
	M (\rho_i; u)
	\ge
	(1 - \beta^{1/2})
	\beta^{1/2}
	\esssup_{
		B_r^{y_i}
	} 
	v_i.
$$
Thus, formula~\eqref{PL3.1.6} enable us to assert that
\begin{equation}
	\esssup_{
		B_r^{y_i}
		\cap
		\Omega
	}
	u
	-
	M (\rho_i; u)
	\ge
	\tau
	r^{p / (p - 1)}
	f^{1 / (p-1)} (s, \beta^{1/2} M (\rho_i; u)),
	\label{PL3.1.7}
\end{equation}
where the constant $\tau > 0$ depends only on $n$, $p$, $C_1$, $C_2$, and $\beta$.
Now, assume that~\eqref{PL3.1.1} is not valid,
then there exists a real number $\zeta > 0$ for which
\begin{equation}
	\esssup_{
		B_r^{y_i}
		\cap
		\Omega
	}
	u
	+
	\zeta
	<
	(2 - \beta^{1/2})
	M (\rho_i; u).
	\label{PL3.1.9}
\end{equation}
We denote
$
	\omega_i
	=
	\{
		x \in B_r^{y_i} \cap \Omega
		:
		u (x) 
		> 
		2 
		M (\rho_i; u)
		- 
		\esssup\nolimits_{
			B_r^{y_i}
			\cap
			\Omega
		}
		u
		-
		\zeta
	\}
$
and
$$
	v_i (x)
	=
	\left\{
		\begin{array}{ll}
			u (x)
			-
			2 
			M (\rho_i; u) 
			+ 
			\esssup\nolimits_{
				B_r^{y_i}
				\cap
				\Omega
			}
			u
			+
			\zeta,
			&
			x \in \omega_i
			\\
			0,
			&
			x \in B_r^{y_i} \setminus \omega_i.
		\end{array}
	\right.
$$
As above, the function $v_i$ satisfies inequality~\eqref{PL3.1.8}.
From~\eqref{PL3.1.9}, it follows that
$
	u (x) > \beta^{1/2} M (\rho_i; u)
$
for all $x \in \omega_i$.
Hence, in accordance with~\eqref{1.5} the function $v_i$ also satisfies 
inequality~\eqref{PL3.1.2}.
Consequently, repeating the previous arguments, we again obtain~\eqref{PL3.1.6}.
Further, it presents no special problems to verify that 
$$
	\esssup_{
		B_r^{y_i}
	} 
	v_i
	=
	2
	\left(
		\esssup_{
			B_r^{y_i}
			\cap
			\Omega
		}
		u
		-
		M (\rho_i; u)
	\right)
	+
	\zeta.
$$
In addition, formula~\eqref{PL3.1.5} implies the estimate
$$
	\lim_{\varepsilon \to +0}
	\esssup_{
		B_\varepsilon^{y_i}
	} 
	v_i
	\ge
	\esssup_{
		B_r^{y_i}
		\cap
		\Omega
	}
	u
	-
	M (\rho_i; u)
	+
	\zeta
	>
	0.
$$
Therefore, we have
$$
	\lim_{\varepsilon \to +0}
	\esssup_{
		B_\varepsilon^{y_i}
	} 
	v_i
	\ge
	\frac{1}{2}
	\esssup_{
		B_r^{y_i}
	}
	v_i
	>
	0
$$
and Lemma~\ref{L4.4} enables us to assert that
$\mes (\omega_i \cap B_{r/4}^{y_i}) \ge \delta r^n$,
where the constant $\delta > 0$ depends only on $n$, $p$, $C_1$, and $C_2$.
Thus, 
$$
	2
	\left(
		\esssup_{
			B_r^{y_i}
			\cap
			\Omega
		}
		u
		-
		M (\rho_i; u)
	\right)
	+
	\zeta
	\ge
	(\varkappa \delta)^{1 / (p  -1)}
	r^{p / (p - 1)}
	f^{1 / (p  -1)} (s, \beta^{1/2} M (\rho_i; u))
$$
by inequality~\eqref{PL3.1.6}.
Finally, passing to the limit in the last expression as $\zeta \to +0$,
we derive~\eqref{PL3.1.7} once more.

Since
$$
	M (r_1; u) - M (r_0; u)
	\ge
	\esssup_{
		B_r^{y_1}
		\cap
		\Omega
	}
	u
	-
	M (\rho_1; u)
$$
and $M (\rho_1; u) \ge M (r_0; u) \ge \beta^{1/2} M (r_1; u)$,
relation~\eqref{PL3.1.7} with $i = 1$ proves the lemma in the case of 
$r = (r_1 - r_0) / 4$.
For $r = 1 / \lambda$, using~\eqref{PL3.1.7}, we obtain
$$
	\sum_{i=1}^N
	\left(
		\esssup_{
			B_r^{y_i}
			\cap
			\Omega
		}
		u
		-
		M (\rho_i; u)
	\right)
	\ge
	\tau
	N 
	r
	\lambda^{- 1 / (p - 1)}
	f^{1 / (p-1)} (s, \beta M (r_1; u)),
$$
whence in accordance with the inequalities 
$N r \ge (r_1 - r_0) / 4$,
$$
	M (\rho_{i + 1}; u)
	\ge
	\esssup_{
		B_r^{y_i}
		\cap
		\Omega
	}
	u,
	\quad
	i = 1, \ldots, N - 1,
$$
and
$$
	M (r_1; u)
	\ge
	\esssup_{
		B_r^{y_N}
		\cap
		\Omega
	}
	u
$$
it follows that
$$
	M(r_1; u) - M (\rho_1; u)
	\ge
	\frac{1}{4}
	\tau
	(r_1 - r_0)
	\lambda^{- 1 / (p - 1)}
	f^{1 / (p-1)} (s, \beta M (r_1; u)).
$$
Lemma~\ref{L3.1} is completely proved.
\end{proof}

\end{document}